\documentclass[reqno]{amsart} % placeholder para CJM; cambiar a \documentclass{cjm} al enviar

\usepackage{amsmath, amssymb, amsthm}
\usepackage{mathtools}
\usepackage{enumitem}
\usepackage{hyperref}
\hypersetup{colorlinks=true, linkcolor=black, citecolor=black, urlcolor=black}
\usepackage{array}
\usepackage{caption}
\usepackage{listings}
\usepackage[toc,page]{appendix}
\usepackage{booktabs}
\usepackage{graphicx}
\usepackage{multirow}
\usepackage{bm}
\usepackage{physics}
\usepackage{multicol}
\usepackage{doi}

\usepackage{dutchcal}
\DeclareMathAlphabet{\pazocal}{OMS}{zplm}{m}{n}

\numberwithin{equation}{section}
\theoremstyle{plain}
\newtheorem{theorem}{Theorem}[section]
\newtheorem{lemma}[theorem]{Lemma}
\newtheorem{proposition}[theorem]{Proposition}

\theoremstyle{definition}
\newtheorem{definition}[theorem]{Definition}
\newtheorem{example}[theorem]{Example}
\theoremstyle{remark}
\newtheorem{remark}[theorem]{Remark}

\newcommand{\Deltauno}{\Delta_1}
\newcommand{\Deltados}{\Delta_2}

\title[Derivations in Dialgebras]{Derivations and Biderivations in Dialgebras}

\author{Gabriel Gustavo Restrepo-S\'{a}nchez}

\author{Jos\'e Gregorio Rodr\'{i}guez-Nieto}

\author{Olga Patricia Salazar-D\'{i}az}

\author{Andr\'{e}s Sarrazola-Alzate}

\author{Ra\'{u}l Vel\'asquez}

\date{\today}

\begin{document}

\begin{abstract}
The concepts of derivations and right derivations for Leibniz algebras and $K$-B quasi-Jordan algebras naturally arise from the inner derivations determined by their algebraic structures. In this paper we introduce the corresponding analogues for dialgebras, which we call diderivations, and examine their properties in relation to antiderivations and right derivations. Our approach is based on the study of multiplicative operators and on the construction of the Leibniz algebra generated by biderivations, thereby providing a systematic framework that unifies several types of derivation-like operators. In addition to the general theory, we present a complete classification of the spaces of diderivations for dialgebras of dimensions two and three, obtained through explicit computations. These low-dimensional results not only exemplify the general constructions but also reveal structural patterns that inform possible extensions to higher dimensions and more intricate algebraic contexts.
\end{abstract}

\maketitle

% ======== INTRODUCTION ========
\section*{Introduction}

The study of derivations has long played a central role in algebra, not only because they describe intrinsic symmetries of a structure, but also because they govern its deformations and extensions. In the framework of Leibniz algebras, two distinct notions naturally arise: classical derivations and antiderivations. Their interaction leads to the construction of the Leibniz algebra of biderivations, which itself contains the subalgebra of inner biderivations as a cornerstone. A similar picture appears in the category of $K$-$B$ quasi-Jordan algebras, where biderivations have also been investigated \cite{Velasquez2009}. In that setting, the construction involves classical derivations together with right derivations, and in fact, it has been established that these two notions coincide - a phenomenon that echoes the situation observed in Leibniz algebras. In both cases, inner derivations act as the fundamental building blocks for the algebra of biderivations, underscoring their structural importance.  

Dialgebras, introduced by J.~L.~Loday in the early 1990s, offer a natural environment to broaden these ideas. Defined by two bilinear operations tied together by compatibility relations, dialgebras stand at a crossroads between associative and Leibniz structures: the commutator of their operations recovers precisely a Leibniz algebra. This fact situates dialgebras as a universal enveloping context in which Leibniz algebras emerge, much in the same way as associative algebras underpin Lie theory. As a result, dialgebras provide fertile ground for investigating derivations, cohomological invariants, and homological phenomena. The role of derivations within dialgebras is particularly crucial, since they capture symmetries and encode the mechanisms by which such algebras deform and interact with one another.  

Yet, in contrast with Leibniz algebras, the operator theory of dialgebras remains largely undeveloped. To this day, only the classical definition of derivations with respect to both products has been explored, as originally formulated by Loday. The work of Lin and Zhang \cite{Lin2010} represents a first step, focusing on derivations of the polynomial dialgebra $K[x,y]$, but no systematic effort has been devoted to developing analogues of antiderivations or right derivations in this context. This gap in the literature naturally raises the question: can dialgebras support operator structures as rich as those known for Leibniz algebras?  

The present article aims to take a step in this direction. On the one hand, we introduce and study a special type of derivations in dialgebras, namely those defined by multiplicative operators - the inner derivations - and extend the framework to incorporate a new operator class, which we call diderivations. On the other hand, we go beyond their mere definition to examine how derivations and diderivations act on some of the most intrinsic invariants of dialgebra theory: the center, the annihilator, and the bar-unit set. This dual perspective provides not only a broader understanding of how operator spaces shape the internal life of a dialgebra, but also a way of situating these results within the wider landscape of non-associative algebra.  

In this sense, our study highlights the place of dialgebras as a meeting point where familiar algebraic ideas - derivations, inner structures, operator algebras - are reinterpreted and extended. At the same time, the introduction of diderivations opens up new paths of exploration, bridging known constructions in Leibniz algebras with novel phenomena in the dialgebraic world. We hope that these contributions will not only enrich the theory of dialgebras, but also foster further dialogue between algebraic frameworks where symmetries and derivations are at the core.  

\subsection*{Organization of the paper}

This paper is organized as follows. Section~\ref{Sec:preliminaries} is devoted to recalling the basic background on Leibniz algebras and their close relationship with dialgebras. Special emphasis is placed on the role of derivations, antiderivations, and biderivations, as these notions will be fundamental in the constructions developed later. We also revisit the seminal work of Loday~\cite{Loday1993}, which established the connection between Leibniz algebras and dialgebras, thereby providing the conceptual framework in which our results are situated.  

In Section~\ref{Sec:Der_&_Dider}, we introduce the notion of \emph{diderivations}, understood as compatible structures obtained by combining antiderivations and right derivations on a dialgebra. Within this section, we establish two structural results of central importance: first, that the space of derivations $\pazocal{D}er(\mathcal{D})$ naturally forms a Lie subalgebra of $\mathfrak{gl}(\mathcal{D})$ (Lemma \ref{Lem:Der_subliealg}); and second, that the space of inner derivations $\mathcal{I}nn(\mathcal{D})$ constitutes an ideal inside $\pazocal{D}er(\mathcal{D})$ (Theorem \ref{thm:Inn_ideal_Der}). These results provide a solid algebraic foundation for the subsequent classification of diderivations.  

Section~\ref{Sec:Cal_Low} is concerned with the \emph{explicit classification of the vector spaces of diderivations in low-dimensional cases}. More precisely, we present a complete classification for dialgebras of dimensions two and three. This extends and refines the computational approaches proposed in~\cite{Abubakar2014, Rikhsiboev2014}, where algorithms for the determination of derivations were tested in the same range of dimensions. Our results therefore contribute to completing the picture of structural operators in small dialgebras, providing the first systematic description of diderivations in these cases.  

In Section~\ref{Sec:bider}, we turn to the study of \emph{biderivations}. We construct the algebra $\pazocal{B}ider(\mathcal{D})$ and prove that it inherits a nontrivial Leibniz structure. More precisely, we show that for any dialgebra $\mathcal{D}$, the pair
\[
\bigl(\pazocal{B}ider(\mathcal{D}), \langle \bullet, \bullet \rangle \bigr)
\]
is canonically endowed with the structure of a Leibniz algebra (Theorem \ref{thm:Bider_Leib}). This result highlights the robustness of the interaction between dialgebraic operators and Leibniz structures, and situates biderivations as natural algebraic objects in their own right. 

In Section~\ref{Sec:example}, we present a concrete example of a dialgebra and provide the classification of both its derivations and diderivations (Theorem \ref{thm:dider_identity} and Lemma \ref{lem:inner_dider}). More precisely, we analyze the dialgebra introduced by L.~Lin and Y.~Zhang in~\cite{Lin2010}, which is generated by polynomials in two commuting variables $K[x,y]$.  

Finally, in Section~\ref{Sec:conclusions}, we summarize the main classification results, underline the implications of the Leibniz structure on biderivations, and indicate several directions for further research. Among these, we stress the potential extension of our classification methods to higher-dimensional dialgebras, as well as the exploration of homological and cohomological aspects of the algebras $\pazocal{D}er(\mathcal{D})$, $\mathcal{I}nn(\mathcal{D})$, and $\pazocal{B}ider(\mathcal{D})$.

\section{Preliminaries}\label{Sec:preliminaries}
 
 In this section, we present a brief review of Leibniz algebras, focusing on the concepts of derivations, antiderivations, and biderivations. We also revisit the concept of dialgebras and their well-known connection to Leibniz algebras. For more details, see \cite{Loday1993}.

Throughout the remainder of this paper, we work over fields of characteristic  zero.
\begin{definition}
    A \textbf{left Leibniz algebra} $\mathcal{L}$ consists of a $K$-vector space endowed with a bilinear product $[\cdot , \cdot]:\mathcal{L}\times \mathcal{L}\rightarrow \mathcal{L}$, called the \textbf{Leibniz bracket}, satisfying the so-called left Leibniz identity:
 
        \begin{equation*}\label{leib}
            [x,[y,z]]=[[x,y],z]-[[x,z],y],
        \end{equation*} 
    for all $x,y,z \in \mathcal{L}$. Similarly, we define a \textbf{right Leibniz algebra}.
\end{definition}
 
  In this paper, we use the term Leibniz algebra to mean a left Leibniz algebra, and the term Leibniz identity to refer to the corresponding left Leibniz identity. We remark for the reader that if the Leibniz bracket is skew-symmetric then $\mathcal{L}$ is a Lie algebra, because the Jacobi identity becomes the Leibniz identity. In other words, every Lie algebra is, in particular, a Leibniz algebra, which is therefore a skew-symmetric version of the Lie algebra.

\begin{definition}\label{bider}
    Let $\mathcal{L}$ be a Leibniz algebra. A
    {\bf{derivation}} over $\mathcal{L}$ is a linear operator $d:\mathcal{L}\to \mathcal{L}$ that satisfies the identity
        \begin{equation*}
            d([x,y])=[dx,y]+[x,dy],
        \end{equation*}
    for all $x,y\in\mathcal{L}$. Besides, an {\bf{antiderivation}} over $\mathcal{L}$ is a linear operator $D:\mathcal{L}\to \mathcal{L}$ that satisfies the identity
        \begin{equation*}
            D([x,y])=[Dx,y]-[Dy,x]
        \end{equation*}
    for all $x,y\in\mathcal{L}$.
\end{definition} 
 
    If $\mathcal{L}$ is a Lie algebra, then any derivation is an antiderivation and reciprocally.

\begin{definition}
    A \textbf{biderivation} on a Leibniz algebra is a pair $(d,D)$ consisting of a derivation $d$ and an antiderivation $D$, which satisfies the identity:
        \begin{equation*}
            [x,dy]=-[x,Dy]
        \end{equation*}
    for all $x,y\in \mathcal{L}$.
\end{definition}
 
    Let us now consider the adjoint operators and their relations with derivations over Leibniz algebras.

\begin{definition}\label{biderinner}
    Let $\mathcal{L}$ be a Leibniz algebra. For every $x\in \mathcal{L}$ we define linear operators, called \textbf{adjoint operators}, $ad_x : \mathcal{L}\rightarrow \mathcal{L}$ and $Ad_{x}: \mathcal{L}\rightarrow \mathcal{L}$ by
        \begin{equation*}
            ad_x(y)\coloneqq [y,x]
                \hspace{0.3 cm}
                \text{and}
                \hspace{0.3 cm}
            Ad_x(y)\coloneqq [x,y]
        \end{equation*}
    for all $y\in\mathcal{L}$.
\end{definition}

\begin{remark}
    From the definition of the Leibniz bracket, it follows that $ad_x$ is a derivation and $Ad_x$ is an antiderivation. Moreover, the pair $(ad_x,\; Ad_x)$ forms a biderivation of $\mathcal{L}$, which is called an \textbf{inner biderivation} over $\mathcal{L}$ associated with $x$.
\end{remark}
 
    Let $\pazocal{B}ider(\mathcal{L})$ denote the set of biderivations on $\mathcal{L}$. If we define the bracket:
        \begin{equation*}
            [(d,D),(d',D')]:=(dd'-d'd,Dd'-d'D),
        \end{equation*}
    for $(d,D),(d',D')\in \pazocal{B}ider(\mathcal{L})$, then $(\pazocal{B}ider(\mathcal{L}), [\cdot , \cdot])$ becomes a Leibniz algebra and the canonical map
        \begin{equation*}
            \begin{matrix}
                \mathcal{L} & \rightarrow & \pazocal{B}ider(\mathcal{L})\\
                \\
                x & \mapsto & (ad_x, Ad_y)
            \end{matrix}
        \end{equation*}
    is a homomorphism of Leibniz algebras. We will denote the set of inner derivations over $\mathcal{L}$ by $\mathcal{I}nn(\mathcal{L})$. It turns out that $\mathcal{I}nn(\mathcal{L})$ is a Leibniz subalgebra of the Leibniz algebra $\pazocal{B}ider(\mathcal{L})$.

\begin{definition}
    A \textbf{diassociative algebra}, or a \textbf{dialgebra} for short, is a $K$-vector space $\mathcal{D}$ endowed with two associative linear applications
        \begin{equation*}
            \vdash : \mathcal{D}\otimes \mathcal{D} \rightarrow \mathcal{D}
                \hspace{0.3 cm}
                \text{and}
                \hspace{0.3 cm}
            \dashv : \mathcal{D}\otimes \mathcal{D} \rightarrow \mathcal{D}
        \end{equation*}
 
    satisfying the following properties:
        
        \begin{equation}\label{D3}
            x\dashv (y \dashv z)=x\dashv (y \vdash z),
        \end{equation}
        \begin{equation}\label{D4}
            (x\dashv y )\vdash z=(x\vdash y) \vdash z,
        \end{equation}
        \begin{equation}\label{D5}
            x\vdash (y \dashv z)=(x\vdash y) \dashv z,
        \end{equation}
    for all $x,y,z\in \mathcal{D}$.
\end{definition}

\begin{remark}
    In general, the identity \( x \dashv (y \vdash z) = (x \dashv y) \vdash z \) does not hold, in contrast to the equality given in equation~\ref{D5} of the previous definition.

\end{remark}

\begin{definition}
    We say that $e\in \mathcal{D}$ is a bar unit if $e\vdash x = x\dashv e = x$, for all $x\in \mathcal{D}$.
\end{definition}
 
    A bar unit is not necessarily unique. The set of all bar units is called the \textbf{halo} and is denoted by $\mathcal{H}(\mathcal{D})$. A \textbf{unital dialgebra} is a dialgebra with a fixed bar unit $e$. Since from equation \ref{D5} if $e \dashv x=x$ or $x \vdash e=x$, for $x\in \mathcal{D}$, then $\vdash=\dashv$, and therefore $\mathcal{D}$ is an associative algebra, we will always consider units on the bar side.

\begin{example} (Dialgebras from differentiable algebras)
    Let $(\pazocal{A}, d)$ be a differential associative algebra, that is, an associative algebra $\pazocal{A}$ together with a linear map $d:\pazocal{A}\rightarrow \pazocal{A}$ such that $d(ab) = d(a)b+ad(b)$, for all  $a,b\in\pazocal{A}$, and $d^2=0$. Defining over $\pazocal{A}$ the products $x\dashv y = xd(y)$ and $x\vdash y = d(x)y$, then $(\pazocal{A},\vdash,\dashv)$ is a dialgebra. If the associative algebra $\pazocal{A}$ is unital, with unit $1$, then $d(1)=0$, which means that $1$ is not a bar unit of $(\pazocal{A},\vdash, \dashv)$. A bar unit in such dialgebra is an element $x\in \pazocal{A}$ such that $d(x) = 1$, but this element does not necessarily exist in $\pazocal{A}$.
\end{example}

\begin{example}
 ($\varphi$-dialgebras) Let $V$ be a $K$-vector space and $\varphi\in V^*\setminus\{0\}$ a  functional in the algebraic dual of $V$. There exists a dialgebra structure on $V$ induced by $\varphi$ (\cite[Lemma 2.1]{Ongay2007}). This is defined by:
		\begin{eqnarray*}
			\nu\vdash \omega\coloneqq \varphi(\nu)\omega
			\hspace{0.5 cm}
			\text{ and }
			\hspace{0.5 cm}
			\nu\vdash \omega\coloneqq \nu\varphi(\omega).
		\end{eqnarray*}
	Such a dialgebra is called a $\varphi$-dialgebra and it is denoted by $V_{\varphi}$.
\end{example}

\begin{example} (Dialgebras from group actions)
    Let $f(x)\in K[x]$. Suppose that $f(x)$ is not irreducible with a root $\alpha\in K$, and let $g(x)$ be an irreducible factor of $f(x)$. Let us consider $K_{f}$ and $K_g$ the splitting fields of $f$ and $g$, respectively, in such a way that we have a tower of extensions $K\subseteq K_g\subseteq K_f$. We assume $\mathcal{R}(g)$ is the set of roots of $g$ and $\mathcal{R}_{K_f}(g)$ the $K_f$-vector space generated by $\mathcal{R}(g)$. The Galois group $\mathcal{G}al(K_g/K)$ acts transitively on $\mathcal{R}(g)$ and by the isomorphism extension theorem, the group $\mathcal{G}al(K_f/K)$ does it as well.
	
	Let $V=Vect_{K_f}\langle \mathcal{R}(g),\alpha \rangle$. Then $\alpha$ is a fixed point and, by the previous paragraph,  $\mathcal{G}al(K_f/K)$ acts transitively on $V\setminus \{\alpha\}=\mathcal{R}_{K_f}(g)$. By \cite[Example 4.2]{Kinyon2007} $\mathcal{D}\coloneqq V\times\mathcal{G}al(K_f/K)$ carries the natural dialgebra structure:
        \begin{equation*}
			(\nu, g)\vdash (\omega,h)\coloneqq (g\omega, gh)
			\hspace{0.5 cm}
			\text{ and }
			\hspace{0.5 cm}
			(\nu, g)\vdash (\omega,h) \coloneqq (\nu, gh)
		\end{equation*}
    for all $\nu,\omega \in V$ and $h,g\in \mathcal{G}al(K_f/K)$.
\end{example}
 
    We now introduce some structural aspects of dialgebras, as well as the notions of ideals and subdialgebras, which are analogous to those in classical algebraic structures.
\begin{definition}
    A \textbf{two-sided ideal} $\pazocal{I}$ of a dialgebra $\mathcal{D}$ is a subspace of $\mathcal{D}$ such that $x\vdash y$ and $x\dashv y$ are in $\pazocal{I}$, wherever $x\in\pazocal{I}$ or $y\in\pazocal{I}$. It is clear that the quotient $\mathcal{D}/\pazocal{I}$ has a canonical dialgebra structure, and the kernel of any morphism of dialgebras is an ideal of the source dialgebra.
 
    We say that $\mathcal{B}$ is a \textbf{subdialgebra} if 
    $x\vdash y \in \mathcal{B}$ and $x\dashv y \in \mathcal{B}$, for all $x,y\in \mathcal{B}$.
\end{definition}
 
    Let us now review the definition of the annihilator of a dialgebra, which will play a fundamental roll in this work.

\begin{definition} \label{Dann}
    Let $\mathcal{D}$ be a dialgebra. Let us consider the vector subspace generated by all the elements of the form $a\dashv b - a\vdash b$, for all $a,b\in \mathcal{D}$. We denote this subspace by $\mathcal{D}^{ann}$, i.e.,
        \begin{equation*}
            \mathcal{D}^{ann}\coloneqq \langle a\dashv b - a\vdash b|\, a,b\in \mathcal{D} \rangle
        \end{equation*}
    and we call it the \textbf{annihilator} of $\mathcal{D}$.
\end{definition}
 
    In \cite{Loday2001} J. L. Loday shows that $\mathcal{D}^{ann}$ is an ideal of $\mathcal{D}$. Moreover, the quotient dialgebra $\mathcal{D}_{as}:=\mathcal{D}/\mathcal{D}^{ann}$ is actually an associative algebra, and $\mathcal{D}^{ann}$ is the smallest ideal that satisfies this property. On the other hand, if $\pazocal{A}$ is an associative algebra, then $\pazocal{A}^{ann} = (0)$ and therefore $(\pazocal{A}_{Di})_{as} = \pazocal{A}$, where $\pazocal{A}_{Di}$ is the canonical dialgebra induced by $\pazocal{A}$. Thus, we may conclude that the functor $Di: \textbf{As}\rightarrow \textbf{Dias}$, where $Di(\pazocal{A})=\pazocal{A}_{Di}$, admits a left adjoint functor $as:\textbf{Dias}\rightarrow \textbf{As}$. In fact, let $\mathcal{D}$ be a dialgebra and let $\pazocal{A}$ be an associative algebra. The dialgebra $\pazocal{A}_{Di}$ associated to $\pazocal{A}$, is defined by the following operators:
\[
x \dashv y = xy = x \vdash y \quad \text{for all } x,y \in \pazocal{A}.
\]

Let $\varphi: \mathcal{D} \rightarrow \pazocal{A}_{Di}$ be a homomorphism of dialgebras. Since $\pazocal{A}_{Di}$ satisfies $x \dashv y = x \vdash y$, the image of $\varphi$ must satisfy:
\[
\varphi(x \dashv y - x \vdash y) = \varphi(x) \dashv \varphi(y) - \varphi(x) \vdash \varphi(y) = 0.
\]
Thus, $\mathcal{D}^{ann} \subseteq \ker \varphi$. Then there exists a unique map $\hat{\varphi}: \mathcal{D}_{as} \rightarrow \pazocal{A}$ such that $\varphi = \iota \circ \hat{\varphi}$, where $\iota: \mathcal{D} \to \mathcal{D}_{as}$ is the canonical projection. In $\mathcal{D}_{as}$, the multiplication is given by $[x][y] := [x \dashv y] = [x \vdash y]$, for all $x,y\in \mathcal{D}$. Since $\varphi$ is a dialgebra morphism, we have that
\[
\hat{\varphi}([x][y]) = \hat{\varphi}([x \dashv y]) = \varphi(x \dashv y) = \varphi(x) \dashv \varphi(y) = \varphi(x) \cdot \varphi(y) = \hat{\varphi}([x]) \cdot \hat{\varphi}([y]),
\]
so $\hat{\varphi}$ is an algebra homomorphism. This construction yields the map
\[
\Phi: \operatorname{Hom}_{\mathbf{Dias}}(\mathcal{D}, \pazocal{A}_{Di}) \longrightarrow \operatorname{Hom}_{\mathbf{As}}(\mathcal{D}_{as}, \pazocal{A}), \quad \varphi \mapsto \hat{\varphi}.
\]

Reciprocally, if $\psi: \mathcal{D}_{as} \rightarrow \pazocal{A}$ is an algebra homomorphism, then $\tilde{\psi}: \mathcal{D} \rightarrow \pazocal{A}_{Di}$ defined by
\[
\tilde{\psi}(x) := \psi([x]).
\]
is a dialgebra homomorphism, this is because the operations in $\pazocal{A}_{Di}$ are defined identically as multiplications in $\pazocal{A}$. This produces the inverse map of $\Phi$, so we have a bijection
\[
\operatorname{Hom}_{\mathbf{Dias}}(\mathcal{D}, \pazocal{A}_{Di}) \cong \operatorname{Hom}_{\mathbf{As}}(\mathcal{D}_{as}, \pazocal{A}).
\]

This natural bijection proves that the functor
\[
as: \mathbf{Dias} \rightarrow \mathbf{As}, \quad \mathcal{D} \mapsto \mathcal{D} / \mathcal{D}^{ann}
\]
is left adjoint to the functor $Di: \mathbf{As} \rightarrow \mathbf{Dias}$. See \cite{Loday2001} for more details. 

    We now present an example of how dialgebras can be constructed using the concept of bimodules. Indeed, it can be shown that every dialgebra arises in this way. Thereby, dialgebras can be classified in terms of an associative algebra equipped with a bimodule structure. See \cite{Frabetti2001} for further details.

\begin{example}\label{dimodular}
    Let $\pazocal{A}$ be an associative algebra and $\pazocal{M}$ be a bimodule. If $f:\pazocal{M}\rightarrow\pazocal{A}$ is an $\pazocal{A}$-bimodular application, then we can define a dialgebra structure on $\pazocal{M}$ by declaring
        \begin{equation*}
            m\dashv n:=mf(n)
                \hspace{0.3 cm}
                \text{and}
                \hspace{0.3 cm}
            m\vdash n:=f(m)n
        \end{equation*}
    for all $m,n\in\pazocal{M}$.
\end{example} 
 
 In general, we have the following proposition.   

\begin{proposition}\label{Frabetti}
    Let $\mathcal{D}$ be a dialgebra and $\mathcal{D}_{as}$ be its canonical associative algebra. There exist a $\mathcal{D}_{as}$-bimodule structure over $\mathcal{D}$ and a morphism of $\mathcal{D}_{as}$-bimodules $\varphi:\mathcal{D}\rightarrow\mathcal{D}_{as}$ such that the dialgebra structure of $\mathcal{D}$ can be recovered by defining $a\dashv b \coloneqq a\cdot \varphi(b)$ and $a\vdash b = \varphi(a)\cdot b$.
\end{proposition}

Finally, let $\mathcal{D}$ be a dialgebra. We will denote by $Z_{B}(\mathcal{D})$ the subspace:
        \begin{equation*}
            Z_{B}(\mathcal{D})\coloneqq\left\{z\in \mathcal{D} \;|\; z\vdash x=0=x\dashv z , \text{ for all } x\in \mathcal{D}\right\}.
        \end{equation*}
        
We have the following results.

\begin{lemma}\cite[Page 193]{Velasquez2009}\label{anne}
    Let $\mathcal{D}$ be a dialgebra, then 
        \begin{itemize}
            \item[(a)] $\mathcal{D}^{ann}$ and $Z_B(\mathcal{D})$ are ideals of $\mathcal{D}$,
            \item[(b)] $\mathcal{D}*Z_{B}(\mathcal{D})\subseteq \mathcal{D}^{ann}$ and $Z_{B}(\mathcal{D})*\mathcal{D}\subseteq\mathcal{D}^{ann}$. 
        \end{itemize}
    Here $*$ denotes one of the products $\vdash$ or $\dashv$.
\end{lemma}

\begin{lemma} \cite[Theorem 9 $\&$ Lemma 10]{Velasquez2009}\label{annzb}
    Let $e$ be a bar unit of $\mathcal{D}$, then 
        \begin{equation*}
            \mathcal{D}^{ann}=Z_{B}(\mathcal{D})=\left\{z\in \mathcal{D} \;|\; e\dashv z=0=z\vdash e \right\}
        \end{equation*}
    and 
        \begin{equation*}
            \mathcal{H}(\mathcal{D}) = \left\{ e+x\;|\; x\in \mathcal{D}^{ann} \right\}.
        \end{equation*}
\end{lemma}
 
    The functorial relation between the category of associative algebras and the category of Lie algebras can be extended to the categories of dialgebras and Leibniz algebras. Actually, we have the following analogue of the canonical Lie bracket defined from an associative algebra in the context of dialgebras and Leibniz algebras.

\begin{lemma}\cite{Frabetti2001} \label{di-li}
    Let $\mathcal{D}$ be a dialgebra. The bilinear map
        \begin{equation*}
            [a,b]:=a\dashv b-b\vdash a
        \end{equation*}
    defines over $\mathcal{D}$ a Leibniz algebra structure. This canonical Leibniz algebra associated to $\mathcal{D}$ will be denoted by $\mathcal{D}_{\mathcal{L}}$.
\end{lemma}

\begin{remark}
    The Leibniz bracket defined in the previous lemma is skew-symmetric if $\mathcal{D}$ is an associative algebra.
\end{remark}
 
    From the previous lemma we have a functor
    $(\cdot)_{\mathcal{L}}:{\bf{Dias}}\to {\bf{Leib}}$
    from the category \textbf{Dias} of dialgebras to the category \textbf{Leib} of Leibniz algebras. Considering the classical functor $(\cdot)_{Lie}:{\bf{As}}\to {\bf{Lie}}$, from the category \textbf{As} of associative algebras to the category \textbf{Lie} of Lie algebras, we get the following commutative diagram of functors.

\begin{proposition}\cite{LodayPirashvili1993}
    The following categorical diagram is commutative 
        \begin{equation*}
                \begin{array}{ccc}
                {\bf{Dias}} & \stackrel{(\cdot)_{\pazocal{L}}}{\rightarrow}& {\bf{Leib}} \\
                \uparrow & & \uparrow \\
                {\bf{As}}&\stackrel{(\cdot)_{Lie}}{\rightarrow} & {\bf{Lie}}.
            \end{array}
        \end{equation*}
\end{proposition}

\section{Derivations and Diderivations over Dialgebras}\label{Sec:Der_&_Dider}
 
    In this section, we construct a compatible structure involving antiderivations and right derivations on a dialgebra. These objects are referred to as \textbf{diderivations}. We provide a characterization of both derivations and diderivations, while the Leibniz algebra of biderivations will be introduced in the next section.

\subsection{Derivations over Dialgebras}
 
    In this subsection, we investigate the concept of derivations on a dialgebra $\mathcal{D}$. Extending the classical notion from associative algebras.
\begin{definition}
    A \textbf{derivation} over $\mathcal{D}$ is a linear map $d:\mathcal{D}\rightarrow\mathcal{D}$ such that 
        \begin{equation*}
            d(a*b) = d(a)*b + a*d(b)
        \end{equation*}
    for all $a,b\in\mathcal{D}$, where $*$ denotes one of the products $\vdash$ or $\dashv$.
\end{definition}
 
    We will denote by $\pazocal{D}er(\mathcal{D})$ the vector space generated by all the derivations $d:\mathcal{D}\rightarrow\mathcal{D}$ defined on $\mathcal{D}$. In order to show the existence of derivations over any dialgebra, we consider the linear map $ad_a:\mathcal{D}\rightarrow\mathcal{D}$ defined by $ad_a\coloneqq R_a^{\dashv}-L_{a}^{\vdash}$, where $R_a^{\dashv}$ and $L_a^{\vdash}$ are the multiplicative operators $R_a^{\dashv}(b)=b\dashv a$ and $L_a^{\vdash}(b)=a\vdash b$, respectively. 

\begin{lemma}
    For all $a\in\mathcal{D}$, $ad_a \in \pazocal{D}er(\mathcal{D})$.
\end{lemma}
\begin{proof}
    We demonstrate the Leibniz rule for the product $\vdash$; the proof for $\dashv$ follows in a similar fashion. In fact, let $b, c \in \mathcal{D}$, then
        \begin{align*}
            ad_a(b)\vdash c+b\vdash ad_a(c)&=\left(R^{\dashv}_a(b)-L^{\vdash}_a(b)\right)\vdash c +b\vdash\left( R^{\dashv}_a(c)-L^{\vdash}_a(c)\right)\\
            &=L^{\vdash}_{(R^{\dashv}_a(b)-L^{\vdash}_a(b))}(c)+L^{\vdash}_b({R^{\dashv}_a(c)-L^{\vdash}_a(c)})\\
            &=L^{\vdash}_{R^{\dashv}_a(b)}(c)-L^{\vdash}_{L^{\vdash}_a(b)}(c)+ L^{\vdash}_bR^{\dashv}_a(c)  -L^{\vdash}_bL^{\vdash}_a(c).
        \end{align*}
        Besides, since 
        \begin{enumerate}
            \item[(i)] $L^{\vdash}_{L^{\vdash}_a(b)}(c)=(a\vdash b)\vdash c=a\vdash (b\vdash c)=L_{a}^{\vdash}L_{b}^{\vdash}(c)$,
            \item[(ii)]  $L^{\vdash}_b R^{\dashv}_a(c)=b\vdash (c\dashv a)=(b\vdash c)\dashv a=R^{\dashv}_a L^{\vdash}_{b}(c)$, and 
            \item[(iii)] $L^{\vdash}_b L^{\vdash}_a(c)=b\vdash (a \vdash c)=(b\vdash a)\vdash c=(b \dashv a)\vdash c=L^{\vdash}_{R^{\dashv}_{a}(b)}(c)$, we have that 
        \end{enumerate}
        \begin{align*}
            ad_a(b)\vdash c+b\vdash ad_a(c)&=L^{\vdash}_{R^{\dashv}_a(b)}(c)-L^{\vdash}_aL^{\vdash}_b(c)+R^{\dashv}_aL^{\vdash}_b(c)-L^{\vdash}_{R^{\dashv}_a(b)}(c)\\
            &=-L^{\vdash}_aL^{\vdash}_b(c)+R^{\dashv}_aL^{\vdash}_b(c)\\
            &=R^{\dashv}_a(b\vdash c)-L^{\vdash}_a(b\vdash c)\\
            &=ad_a(b\vdash c).
        \end{align*}
\end{proof}
 
    The derivations $ad_a$, $a\in \mathcal{D}$, are called \textbf{inner derivations} of $\mathcal{D}$ and we will denote by $\mathcal{I}nn(\mathcal{D})=\langle ad_a\; |\; a\in\mathcal{D}\rangle$ the vector space of inner derivations. As we will show, this is an ideal of $\pazocal{D}er(\mathcal{D})$ and a Lie subalgebra of $\mathfrak{gl}(\mathcal{D}) = End(\mathcal{D})$ (Lemma \ref{Lem:Der_subliealg}).

\begin{theorem}\label{derL*}
    Let $T:\mathcal{D}\rightarrow\mathcal{D}$ be a linear map, then $T$ is a derivation if and only if $L^{*}_{T(a)}=[T,L^{*}_a]$, for all $a\in\mathcal{D}$, where $*$ represents one of the products $\vdash$ or $\dashv$.
\end{theorem}

\begin{proof}
    Let $T$ be a derivation over $\mathcal{D}$. Since for all $a,b\in \mathcal{D}$ it holds
        \begin{equation*}
            T(L^{*}_a(b))=T(a*b)=T(a)\ast b+a\ast T(b) = L^{\ast}_{T(a)}(b)+L^{\ast}_a(T(b)),
        \end{equation*}
    then 
    \begin{equation*}   
        L^{\ast}_{T(a)}(b)=T(L^{\ast}_a(b))-L^{\ast}_a(T(b))= [T,L^{\ast}_a](b)
    \end{equation*}
    and therefore $L^{\ast}_{T(a)}=[T,L^{\ast}_a]$ for all $a\in \mathcal{D}$.
 
    On the other hand, if $L^{\ast}_{T(a)}=[T,L^{\ast}_a]$ for all $a\in \mathcal{D}$, then 
        \begin{equation*}
            T(a)*b=L^{\ast}_{T(a)}(b)=T(L^{\ast}_a(b))-L^{\ast}_a(T(b))=T(a*b)-a*T(b),
        \end{equation*}
    for all $b\in\mathcal{D}$, in other words $T(a*b)=T(a)*b+a*T(b)$, for all $a,b\in \mathcal{D}$. This implies that $T$ is a derivation.
\end {proof}
 
    In the same way, we have the following result.

\begin{theorem}\label{derR*}
    Let $T:\mathcal{D}\rightarrow\mathcal{D}$ be a linear map, then $T$ is a derivation if and only if $R^{*}_{T(a)}=[T,R^{*}_a]$, for all $a\in\mathcal{D}$, where $*$ represents one of the products $\vdash$ or $\dashv$. 
\end{theorem}

\begin{lemma}\label{Lem:Der_subliealg}
$\pazocal{D}er(\mathcal{D})$ is a Lie subalgebra of $\mathfrak{gl}(\mathcal{D})$.
\end{lemma}

\begin{proof}
    Let $d_1,d_2\in \pazocal{D}er(\mathcal{D})$. Then for all $a\in\mathcal{D}$,
        \begin{align*}
            L^*_{\left[d_1,d_2\right](a)}&=L^*_{d_1\left(d_2(a)\right)}-L^*_{d_2\left(d_1(a)\right)}=\left[d_1,L^*_{d_2(a)}\right]-\left[d_2,L^*_{d_1(a)}\right]\\
            &=\left[d_1,\left[d_2,L^*_a\right]\right]-\left[d_2,\left[d_1,L^*_a\right]\right]\\
            &=\left[\left[d_1,d_2\right],L^*_a\right].
        \end{align*}
    Thus, from Theorem \ref{derL*}, we have that $\left[d_1,d_2\right]\in \pazocal{D}er(\mathcal{D})$.
\end{proof}

\begin{theorem}\label{thm:Inn_ideal_Der}
$\mathcal{I}nn(\mathcal{D})$ is an ideal of  $\pazocal{D}er(\mathcal{D})$.
\end {theorem}

\begin{proof}
    Let $ad_a\in \mathcal{I}nn(\mathcal{D})$ and $d\in \pazocal{D}er(\mathcal{D})$, then
        \begin{align*}
            [d,ad_a] &=[d,(R^{\dashv}_a-L^{\vdash}_a)]=[d,R^{\dashv}_a]-[d,L^{\vdash}_a]\\
            &=R^{\dashv}_{d(a)}-L^{\vdash}_{d(a)}=ad_{d(a)}.
        \end{align*}
    Therefore, $[d,ad_a]\in \mathcal{I}nn(\mathcal{D})$.
\end{proof}
 
    Let $\mathcal{D}$ be a dialgebra, and let $\mathcal{D}_{\pazocal{L}}$ be the corresponding Leibniz algebra generated by $\mathcal{D}$ via the bracket operation $[b, a] = b \dashv a - a \vdash b$. Then $\operatorname{ad}_a(b) = [b, a]$ is an inner derivation over $\mathcal{D}_{\pazocal{L}}$. Hence, we may conclude that inner derivations over $\mathcal{D}$ give rise to inner derivations over $\mathcal{D}_{\pazocal{L}}$.

    Let us now study the action of $\pazocal{D}er(\mathcal{D})$ over bar units, and the ideals $\mathcal{D}^{ann}$ and $Z_B(\mathcal{D})$.

\begin{lemma}
    Let $d$ be a derivation of a dialgebra $\mathcal{D}$, then

    \begin{itemize}
	\item[(a)] $d(\mathcal{D}^{ann})\subseteq             \mathcal{D}^{ann}$,
	\item[(b)] $d(Z_B(\mathcal{D}))\subseteq Z_B(\mathcal{D})$.
	\item[(c)] If $\mathcal{D}$ is unital, then $d\left(\mathcal{H}(\mathcal{D})\right)\subseteq \mathcal{D}^{ann}$.
	\end{itemize}
\end{lemma}

\begin{proof}
    \begin{enumerate}
        \item[(a)] Let us start proving that $d(\mathcal{D}^{ann})\subseteq \mathcal{D}^{ann}$. In fact, 
            \begin{align*}
                d(x\dashv y-x\vdash y)&=d(x\dashv y)-d(x\vdash y)\\
                &=dx\dashv y +x\dashv dy-dx\vdash y -x\vdash dy\\
                &=(dx\dashv y-dx\vdash y) +(x\dashv dy -x\vdash dy)\ \in \mathcal{D}^{ann}.
            \end{align*}

        \item[(b)] On the other hand, given $d\in \pazocal{D}er(\mathcal{D})$ and $z\in Z_B(\mathcal{D})$, then for all $x\in \mathcal{D}$ it holds
            \begin{equation*}
                0=d(0)=d(x\dashv z)=d(x)\dashv z+x\dashv d(z)=x\dashv d(z).
            \end{equation*}
        In a completely similar way, we prove that $d(z)\vdash x=0$, and therefore $d(z)\in Z_B(\mathcal{D})$.

        \item[(c)] Let $e$ be a bar unit in $\mathcal{D}$, then
        \[
        d(e)=d(e\vdash e)=d(e)\vdash e +e\vdash d(e)=d(e)\vdash e + d(e).
        \]
        Thus, $d(e)\vdash e=0$. Similarly, we prove $e\dashv d (e)=0$. Thereby, from Lemma \ref{annzb} it follows that $d(e)\in \mathcal{D}^{ann}$.

    \end{enumerate}
\end{proof}

\subsection{Diderivations over Dialgebras}
 
    Let us now introduce the concept of diderivation of a dialgebra and prove some of its basic properties and its relation with derivations. To begin with, we recall that in every Leibniz algebra $\mathcal{L}$, there exists the concept of an antiderivation, which is a linear map $\delta:\mathcal{L} \rightarrow \mathcal{L}$ such that
        \begin{equation*}
            \delta\left(\left[x,y\right]\right)=\left[\delta(x),y\right]-\left[\delta(y),x\right],
        \end{equation*}
    for all $x,y\in\mathcal{L}$. Moreover, if $\mathcal{L}$ is a Lie algebra then the concepts of derivation and antiderivation are equivalent.
 
    On the other hand, in every Leibniz algebra $\mathcal{L}$ the linear map $Ad_a:\mathcal{L}\rightarrow\mathcal{L}$ defined by $Ad_{x}(y)\coloneqq [x,y]$ is an inner antiderivation of $\mathcal{L}$.
    Assuming that $\mathcal{L}=\mathcal{D}_{\mathcal{L}}$, then in terms of the multiplicative operators we have
        \begin{equation*}
            Ad_a=L^{\dashv}_a - R^{\vdash}_a.
        \end{equation*}
 
    To determine the action of this kind of linear maps on a dialgebra $\mathcal{D}$, let us consider the following slight modification:

        \begin{align*}
            Ad_a : &\ \mathcal{D}  \to\ \mathcal{D} \\
                    & b \to Ad_a(b)=R^{\vdash}_a(b)- L^{\dashv}_a(b)\, ,
        \end{align*}
    Let us see how this transformation acts on the products $b\dashv c$ and $b\vdash c$.
        \begin{align*}
            Ad_a(b\dashv c)&=(R^{\vdash}_a-L^{\dashv}_a)(b\dashv c)\\
            &=(b\dashv c)\vdash a-a\dashv(b\dashv c)\\
            &=b\vdash (c\vdash a)-(a\dashv b)\dashv c-b\vdash (a\dashv c)+(b\vdash a)\dashv c\\
            &=b\vdash (c\vdash a-a\dashv c)+(b\vdash a-a\dashv b)\dashv c\\
            &=b\vdash(R^{\vdash}_a-L^{\dashv}_a)(c)+(R^{\vdash}_a-L^{\dashv}_a)(b)\dashv c\\
            &=Ad_a(b)\dashv c+b \vdash Ad_a(c)
    \end{align*} 
and 
    \begin{align*}
        Ad_a(b\vdash c)&=(R^{\vdash}_a-L^{\dashv}_a)(b\vdash c)=(b\vdash c)\vdash a-a\dashv(b\vdash c)\\
        &=b\vdash (c\vdash a)-(a\dashv b)\dashv c-b\vdash (a\dashv c)+(b\vdash a)\dashv c\\
        &=b\vdash (c\vdash a-a\dashv c)+(b\vdash a-a\dashv b)\dashv c \\
        &=b\vdash(R^{\vdash}_a-L^{\dashv}_a)(c)+(R^{\vdash}_a-L^{\dashv}_a)(b)\dashv c\\
        &=Ad_a(b)\dashv c+b \vdash Ad_a(c).
    \end{align*}
 
    From the previous discussion, the linear map $Ad_a$ satisfies the following definition.

\begin{definition}
    A \textbf{diderivation} of a dialgebra $\mathcal{D}$ is a linear map $\delta:\mathcal{D} \rightarrow \mathcal{D}$ such that
        \begin{equation*}
           \delta(a\ast b)=\delta(a)\dashv b+a \vdash \delta(b) 
        \end{equation*}
    for all $a,b\in\mathcal{D}$ and $*$ being one of the products $\vdash$ or $\dashv$. 
\end{definition}
 
    We will denote by $\pazocal{D}ider(\mathcal{D})$ the vector space generated by all the diderivations of $\mathcal{D}$. This space is not empty because for every $a\in\mathcal{D}$ we already know that $Ad_a\in\pazocal{D}ider(\mathcal{D})$. This kind of diderivations is called \textbf{inner diderivations}, and the vector space generated by such linear maps will be denoted by $\pazocal{D}\mathcal{I}nn(\mathcal{D})$.

\begin{remark}
    If $\mathcal{D}$ is an associative algebra then $\pazocal{D}ider(\mathcal{D})=\pazocal{D}er(\mathcal{D})$. In particular $ad_a = Ad_a$ and therefore $\mathcal{I}nn(\mathcal{D}) = \pazocal{D}\mathcal{I}nn(\mathcal{D})$.
\end{remark}
 
    Let us characterize the diderivations in terms of the multiplicative operators.

\begin{theorem}\label{diderL*} 
    A linear transformation $\delta:\mathcal{D}\rightarrow\mathcal{D}$ is a diderivation if and only if $\delta L^{\ast}_a-L^{\vdash}_a\delta=L^{\dashv}_{\delta(a)}$, for every $a\in\mathcal{D}$, where $*$ represents one of the products $\dashv$ or $\vdash$.
\end {theorem}

\begin{proof}
    If $\delta$ is a diderivation, then for all $a,b\in\mathcal{D}$,
        \begin{align*}
            \delta\left(L^{\ast}_a b\right)=\delta(a\ast b)=\delta(a)\dashv b+a \vdash \delta(b)= L^{\dashv}_{\delta(a)} b+L^{\vdash}_a  \delta(b),
        \end{align*}
    which implies that
        \begin{align*}
            \delta\left(L^{\ast}_a (b)\right) -L^{\vdash}_a  (\delta(b))=L^{\dashv}_{\delta(a)} (b). 
        \end{align*}
    Consequently
    \begin{equation}\label{Mult_dider}
        \delta L^{\ast}_a -L^{\vdash}_a \delta=L^{\dashv}_{\delta(a)}, 
    \end{equation}
    with $\ast$ representing one of the products $\dashv$ or $\vdash$.
 
    Reciprocally, if (\ref{Mult_dider}) holds, then
        \begin{equation*}
            \delta(a)\dashv (b)=L^{\dashv}_{\delta(a)} b=\delta(L^{\ast}_a b) -L^{\vdash}_a  \delta(b)=\delta(a\ast b)-a\vdash\delta(b). 
        \end{equation*}
    This implies the equality $\delta(a\ast b)=\delta(a)\dashv b+a\vdash\delta(b)$, which means that $\delta$ is a diderivation.
\end {proof}
 
Similarly, we obtain the following equivalences.

\begin{theorem}\label{Dider_vs_mult-ops}
    Let $\delta: \mathcal{D}\rightarrow\mathcal{D}$ be a linear map. The following assertions are equivalent.
        \begin{itemize}
            \item[(a)] $\delta$ is a diderivation.
            \item[(b)] $\delta R^{\ast}_a-R^{\dashv}_a\delta=R^{\vdash}_{\delta(a)}$, where $*$ represents one of the products $\vdash$ or $\dashv$.
            \item[(c)] $L^ {\dashv}_{\delta(a)}=[\delta,L^ {\vdash}_a]$ and $R^ {\vdash}_{\delta(a)}=[\delta,R^ {\dashv}_a]$.
        \end{itemize}
\end{theorem}
 
    Using the previous characterizations and the properties of the multiplicative operators it is easy to see that if  $\delta,\delta'\in \pazocal{D}ider(\mathcal{D})$, then
        \begin{equation*}
            [[\delta,\delta'],L_a^{\vdash}]=L_{[\delta,\delta'](a)}^{\vdash}+\left(L_{\delta'(a)}^{\vdash}+L_{\delta'(a)}^{\dashv}\right)\delta-\left(L_{\delta(a)}^{\vdash}+L_{\delta(a)}^{\dashv}\right)\delta{d}'. 
        \end{equation*}
    Therefore, $\pazocal{D}ider(\mathcal{D})$ is generally not a Lie subalgebra of $\mathfrak{gl}(\mathcal{D})$. However, the following property holds.

\begin{theorem}
    If $\mathcal{D}$ is a dialgebra, then $\left[\pazocal{D}ider(\mathcal{D}),\pazocal{D}er(\mathcal{D})\right]\subseteq \pazocal{D}ider(\mathcal{D})$.
\end {theorem}

\begin{proof}
    We will use the properties proved in Theorems \ref{derL*} and \ref{diderL*}.  Let $\delta$ be a diderivation and  $d$ be a derivation of $\mathcal{D}$. We have

        \begin{align*}
            [\delta,d]L^{\ast}_a-L^{\vdash}_a[\delta,d]&=(\delta d)L^{\ast}_a-(d \delta )L^{\ast}_a- L^{\vdash}_a(\delta d)+L^{\vdash}_a( d \delta)\\
            &=\delta(L^{\ast}_{da}+L^{\ast}_ad)-d(L^{\dashv}_{\delta a}+L^{\vdash}_a \delta) \\
            & \hspace{0.4 cm}+(L^{\dashv}_{\delta a}-\delta L^{\ast}_a)d-(dL^{\vdash}_a-L^{\vdash}_{da})\delta\\
            &=\delta L^{\ast}_{da}-dL^{\dashv}_{\delta a}+L^{\dashv}_{\delta a}d-L^{\vdash}_{da}\delta\\
            &=L^{\dashv}_{\delta da}+L^{\vdash}_{da}\tilde{d}-L^{\dashv}_{d\delta a}-L^{\dashv}_{\delta x}d+L^{\dashv}_{\delta x}d-L^{\vdash}_{da}\delta \\
            &=L^{\dashv}_{\delta da}-L^{\dashv}_{d\delta a}\\
            &=L^{\dashv}_{[\delta ,d](a)},
        \end{align*}
    which shows that  $[\delta,d]\in \pazocal{D}ider(\mathcal{D})$.
\end {proof}
 
From the previous theorem we have the following relations.

\begin{lemma}
    For every dialgebra $\mathcal{D}$ we have, $\left[\pazocal{D}\mathcal{I}nn(\mathcal{D}),\pazocal{D}er(\mathcal{D})\right]\subseteq \pazocal{D}\mathcal{I}nn(\mathcal{D})$.  In particular, $\left[\pazocal{D}\mathcal{I}nn(\mathcal{D}),\mathcal{I}nn(\mathcal{D})\right]\subseteq \pazocal{D}\mathcal{I}nn(\mathcal{D})$.
\end{lemma}
 
    Let us finally see how the diderivations act on the ideals $\mathcal{D}^{ann}$ and $Z_B(\mathcal{D})$, and on the halo $\mathcal{H}(\mathcal{D})$, when $\mathcal{D}$ is a unital dialgebra.

\begin{lemma}
    Let $\mathcal{D}$ be a dialgebra  and $\delta\in\pazocal{D}ider(\pazocal{D})$. Then
        \begin{itemize}
            \item[(a)] $\delta(\pazocal{D}^{ann}) = \{0\}$.
            \item[(b)] $x\vdash \delta(z)=\delta(z)\dashv x= 0$ for every $x\in\mathcal{D}$ and $z\in Z_{B}(\mathcal{D})$. 
            \item[(c)] If $\mathcal{D}$ is a unital dialgebra, then $\delta(\mathcal{H}(\mathcal{D})) = \delta(Z_{B}(\mathcal{D})) = \{0\}$. 
        \end{itemize}
\end{lemma}

\begin{proof}
\begin{itemize}
    \item[(a)] This is a straightforward computation following the definitions of $\mathcal{D}^{ann}$ and $\pazocal{D}ider(\mathcal{D})$.
    \item[(b)] Let $x\in\mathcal{D}$ and $z\in Z_{B}(\mathcal{D})$, then $0=\delta (z\vdash x) = \delta(z)\dashv x$. Applying $\delta$ to $x\dashv z$ we get the other property.
    \item[(c)] Let $e\in\mathcal{H}(\mathcal{D})$, then $\delta(e)=\delta(e*e)= \delta(e)+\delta(e)$ and therefore $\delta(e) = 0$.
\end{itemize}
\end{proof}

\section{Diderivations of Low Dimensional Dialgebras}\label{Sec:Cal_Low}
 
    Throughout this section, we will always assume that a dialgebra $\mathcal{D}$  carries a complex vector space structure. In \cite{Abubakar2014, Rikhsiboev2014} the authors introduced an algorithm to compute derivations of dialgebras and tested such an algorithm in low-dimensional cases, this is, on dialgebras of dimension two and three. In this section, we will follow \cite{Abubakar2014} to give a complete classification of the vector space of diderivations for dialgebras of dimensions two and three.
    
\subsection{Diderivations of Two Dimensional Dialgebras}
 
    We recall for the reader that from \cite{Rikhsiboev2014} any complex dialgebra belongs to one of the isomorphism classes given in Table \ref{Two-dimensional-dias}. 

\begin{table}[ht]
    \centering
    \renewcommand{\arraystretch}{1.5} % Increase row height for readability
    \begin{tabular}{| m{5cm} | m{5cm} |}
        \hline
        \textbf{$Dias_{2}^1$} & \textbf{$Dias_{2}^3$} \\
        \hline
        \begin{tabular}{@{}l@{}} 
            $e_1 \vdash e = e_1 \dashv e_1 = e_1$ \\
            $e_2 \dashv e_1 = e_2$
        \end{tabular} 
        & 
        \begin{tabular}{@{}l@{}} 
            $e_1 \vdash e_1 = e_2$ \\
            $e_1 \dashv e_1 = \lambda e_2, \: \lambda \in \mathbb{C}$
        \end{tabular} \\
        \hline
        \textbf{$Dias_{2}^2$} & \textbf{$Dias_{2}^4$} \\
        \hline
        \begin{tabular}{@{}l@{}} 
            $e_1 \vdash e_1 = e_1 \dashv e_1 = e_1$ \\
            $e_1 \dashv e_2 = e_2$
        \end{tabular} 
        & 
        \begin{tabular}{@{}l@{}} 
            $e_1 \vdash e_1 = e_1 \dashv e_1 = e_1$ \\
            $e_1 \vdash e_2 = e_2 \dashv e_1 = e_2$
        \end{tabular} \\
        \hline
    \end{tabular}
    \caption{Two-dimensional complex dialgebras.}
    \label{Two-dimensional-dias}
\end{table}

    Let $\mathcal{D}$ be an $n$-dimensional dialgebra with basis $\left\{e_1,\cdots,e_n\right\}$ and $\delta\in\pazocal{D}ider(\mathcal{D})$. Representing the diderivation $\delta$ in matrix form $\delta = (d_{ij})$ and considering Theorem \ref{Dider_vs_mult-ops} we may construct the system of equations:
        \begin{equation}\label{System_dider}
            [\delta, L_{e_i}^{\vdash}] = \displaystyle\sum_{j=1}^{n}d_{ij} L_{e_j}^{\dashv}
                \hspace{0.5 cm}
                    \text{ and }
                \hspace{0.5 cm}
            [\delta,R_{e_i}^{\dashv}]=\displaystyle\sum_{j=1}^{n}d_{ij}R_{e_j}^{\vdash},
        \end{equation}
    which helps us to reconstruct the diderivation $\delta$.
 
    Let $\mathcal{D}$ be a two-dimensional dialgebra with basis $\{e_1,e_2\}$, then any diderivation can be written in matrix form as
        \begin{equation*}
            \delta =
                \begin{pmatrix}
                    d_{11} & d_{12} \\
                    d_{21} & d_{22}.
                \end{pmatrix}
        \end{equation*}
    In order to find the coefficients of $\delta$ using System (\ref{System_dider}), we first need to find the matrices $R_{e_i}^{\vdash}, R_{e_i}^{\dashv}, L_{e_i}^{\vdash}$ and $L_{e_i}^{\dashv}$.
 
    Let us start considering the dialgebra $Dias_{2}^1$. From Table \ref{Two-dimensional-dias} we have

\begin{multicols}{4}
    \begin{itemize}
        \item[]
        \begin{equation*}
            \resizebox{0.8\columnwidth}{!}{
                $L_{e_1}^{\dashv}=
                    \begin{pmatrix}
                        1 & 0 \\
                        0 & 0
                    \end{pmatrix}$
            }
        \end{equation*}
        \item[]
        \begin{equation*}
            \resizebox{0.8\columnwidth}{!}{
                $L_{e_2}^{\dashv} =
                    \begin{pmatrix}
                        0 & 0 \\
                        1 & 0
                    \end{pmatrix}$
            }
        \end{equation*}
        \item[]
        \begin{equation*}
            \resizebox{0.8\columnwidth}{!}{
                $L_{e_{1}}^{\vdash}=
                    \begin{pmatrix}
                        1 & 0\\
                        0 & 0
                    \end{pmatrix}$
            }
        \end{equation*}
        \item[]
        \begin{equation*}
            \resizebox{0.8\columnwidth}{!}{
                $L_{e_{2}}^{\vdash}=
                    \begin{pmatrix}
                        0 & 0 \\
                        0 & 0
                    \end{pmatrix}$
            }
        \end{equation*}
        \item[]
        \begin{equation*}
            \resizebox{0.8\columnwidth}{!}{
                $R_{e_1}^{\dashv}=
                    \begin{pmatrix}
                        1 & 0 \\
                        0 & 1
                    \end{pmatrix}$
            }
        \end{equation*}
        \item[]
        \begin{equation*}
            \resizebox{0.8\columnwidth}{!}{
                $R_{e_2}^{\dashv} =
                    \begin{pmatrix}
                        0 & 0 \\
                        0 & 0
                    \end{pmatrix}$
            }
        \end{equation*}
        \item[]
        \begin{equation*}
            \resizebox{0.8\columnwidth}{!}{
                $R_{e_{1}}^{\vdash}=
                    \begin{pmatrix}
                        1 & 0\\
                        0 & 0
                    \end{pmatrix}$
            }
        \end{equation*}
        \item[]
        \begin{equation*}
            \resizebox{0.8\columnwidth}{!}{
                $R_{e_{2}}^{\vdash}=
                    \begin{pmatrix}
                        0 & 0 \\
                        0 & 0
                    \end{pmatrix}$
            }
        \end{equation*}
    \end{itemize}
\end{multicols}
 
    which, together with the previous description of the multiplicative operators, shows that $d_{11}=d_{12}=d_{22}=0$, and therefore $dim_{\mathbb{C}}(\pazocal{D}ider(Dias_{2}^1))=1$. 
     
    For the other three isomorphism classes of two-dimensional dialgebras we have used the code explained in Appendix \ref{Latex_code}.

\begin{table}[h!]
\centering
\begin{tabular}{|p{5.5cm}|c|c|}
\hline
\textbf{Diderivation Space} & \textbf{Base} & \textbf{Dimension} \\
\hline
\multirow{4}{5em}{$\pazocal{D}ider(Dias_2^1)$ \\ $\pazocal{D}ider(Dias_2^2)$ \\ $\pazocal{D}ider(Dias_2^3)$} & & \\
 & 
$\begin{pmatrix} 0 & 0 \\ 1 & 0 \end{pmatrix}$ & 
$1$ \\
& & \\
\hline
$\pazocal{D}ider(Dias_2^4)$ & $\begin{pmatrix} 0 & 0 \\ 0 & 0 \end{pmatrix}$ & 0\\
\hline
\end{tabular}
\caption{Isomorphism classes of the space of diderivations in dimension three.}
\label{Iso_classes_dimension_three}
\end{table}

\begin{remark}
    Actually, for $Dias_{3}^2$ we have 
        \begin{equation*}
            \delta = 
                \begin{pmatrix}
                    d_{11} & 0 \\
                    d_{21} & (1+\lambda)d_{11}
                \end{pmatrix}
        \end{equation*}
    with $\lambda =0$ or $\lambda =1$, but if $\lambda = 1$ then $\delta$ is a derivation (\cite[Main Theorem]{Rakhimov2017}) and for $\lambda=0$ we get $d_{11}=0$ and $\delta\in\mathcal{D}ider(Dias_{2}^3)$.
\end{remark}

\subsection{Diderivations of Three Dimensional Dialgebras}
 
    For the sake of the clarity, in Table \ref{Three-dimensional-dias} (Appendix \ref{Latex_code}) we have given the classification of the three dimensional complex dialgebras given in \cite{Rakhimov2017}.

In the following table, we give the dimensions for the vector space of diderivations of three-dimensional dialgebras. The case $Dias_3^{16}$ presents a significantly higher level of difficulty. 
The corresponding dimension calculations are displayed in Table \ref{tab:Dias_3^16}, 
where they are expressed in terms of structural parameters. 
For the reader’s convenience, the detailed computations together with the underlying reasoning 
have been deferred to Appendix \ref{App:Dias_3^16}.

\begin{center}
\begin{tabular}{ |c|c|c| } 
\hline
\textbf{Diderivation Space} & \textbf{Base} & \textbf{Dimension}\\
\hline
 & & \\
\multirow{3}{4em}{$\pazocal{D}ider(Dias_3^1)$} & $\begin{pmatrix}
    0 & 1 & 0 \\
    0 & 0 & 0 \\
    0 & 0 & 0
\end{pmatrix}$ & $1$ \\
 & & \\
\hline
 &  &  \\ 
\multirow{3}{4em}{$\pazocal{D}ider(Dias_3^2)$ \\ $\pazocal{D}ider(Dias_3^3)$} & $\begin{pmatrix}
    0 & 0 & 0 \\
    0 & 0 & 0 \\
    0 & 0 & 0
\end{pmatrix}$ & $0$  \\ 
& &  \\ 
\hline
 &  &  \\ 
\multirow{3}{4em}{$\pazocal{D}ider(Dias_3^4)$ \\ $\pazocal{D}ider(Dias_3^5)$ \\ $\pazocal{D}ider(Dias_3^7)$} & $\begin{pmatrix}
    1 & 0 & 0 \\
    -1 & 0 & 0 \\
    0 & 0 & 0
\end{pmatrix}$ &  $1$  \\ 
& &  \\ 
\hline
{$\pazocal{D}ider(Dias_3^6)$} & $\begin{pmatrix}
    0 & 0 & 0 \\
    0 & 0 & 0 \\
    0 & 0 & 0
\end{pmatrix}$ & $0$ \\
 & & \\
\hline
 & & \\
\multirow{3}{4em}{$\pazocal{D}ider(Dias_3^8)$} & $\begin{pmatrix}
    0 & 0 & 1 \\
    0 & 0 & 0 \\
    0 & 0 & 0
\end{pmatrix}$ & $1$ \\
 & & \\
\hline
 &  &  \\ 
 $\pazocal{D}ider(Dias_3^9)$ & $\begin{pmatrix}
    0 & 0 & 0 \\
    1 & 1 & 0 \\
    0 & 0 & 0
\end{pmatrix}$ & $1$  \\ 
& &  \\ 
\hline
 &  &  \\ 
\multirow{3}{4em}{$\pazocal{D}ider(Dias_3^{10})$ \\ $\pazocal{D}ider(Dias_3^{11})$} & $\begin{pmatrix}
    1 & 0 & 0 \\
    0 & 0 & 0 \\
    0 & 0 & 0
\end{pmatrix}$, $\begin{pmatrix}
    0 & 0 & 1 \\
    0 & 0 & 0 \\
    0 & 0 & 0
\end{pmatrix}$   &  $2$   \\ 
& &  \\ 
\hline
\end{tabular}
\end{center}

\begin{center}
\begin{tabular}{ |c|c|c| } 
\hline
\textbf{Diderivation Space} & \textbf{Base} & \textbf{Dimension}\\
\hline
 & & \\
$\pazocal{D}ider(Dias_3^{12})$& $\begin{pmatrix}
    1 & 0 & 0 \\
    0 & 0 & 0 \\
    0 & 0 & 0
\end{pmatrix}$  & $1$   \\ 
& &  \\
\hline
 & & \\
\multirow{3}{4em}
 &  &  \\ 
\multirow{3}{4em}{$\pazocal{D}ider(Dias_3^{13})$ \\ $\pazocal{D}ider(Dias_3^{14})$} & $\begin{pmatrix}
    1 & 0 & 0 \\
    0 & 0 & 0 \\
    0 & 0 & 0
\end{pmatrix}$, $\begin{pmatrix}
    0 & 0 & 0 \\
    1 & 0 & 0 \\
    0 & 0 & 0
\end{pmatrix}$   &  $2$   \\ 
& &  \\
\hline
 & & \\
$\pazocal{D}ider(Dias_3^{15})$& $\begin{pmatrix}
    0 & 0 & 0 \\
    0 & 0 & 0 \\
    0 & 0 & 0
\end{pmatrix}$  & $0$   \\ 
& &  \\
\hline
 & & \\
$\pazocal{D}ider(Dias_3^{17})$ & $\begin{pmatrix}
    0 & 0 & 0 \\
    0 & 0 & 0 \\
    1 & 0 & 0
\end{pmatrix}$  & $1$   \\ 
& &  \\
\hline
\end{tabular}
\end{center}

\section{Compact case table $Dias_3^{16}$}\label{tab:Dias_3^16}

\small
\renewcommand{\arraystretch}{1}
\begin{center}
\begin{tabular}{|c|c|l|c|}
\hline
\textbf{Branch} & \textbf{Conditions} & \textbf{Description} & \textbf{Dim} \\
\hline
$m\neq 0$ & $\Deltauno,\Deltados\neq 0$ & Generic; only $d_{21},d_{23}$ survive & $2$ \\
\hline
$m\neq 0$ & $\Deltauno=0$, $\Deltados\neq 0$ & One compat.~constraint & $3$ \\
\hline
$m\neq 0$ & $\Deltauno\neq 0$, $\Deltados=0$ & One compat.~constraint & $3$ \\
\hline
$m\neq 0$ & $\Deltauno=\Deltados=0$ & Two compat.~constraints & $4$ \\
\hline
$m\neq 0$ & add $p=-1,q=0$ & \eqref{eq:R3a}(2) drops $\Rightarrow$ +1 dim & $+1$ \\
\hline
$m=0$ & $n+k\neq 0$, $k\neq np$ & Generic in branch & $2$ \\
\hline
$m=0$ & $n+k\neq 0$, $k=np$ & Extra freedom in $(d_{11},d_{33})$ & $3$ \\
\hline
$m=0$ & prev.\ + $p=-1,q=0$ & \eqref{eq:R3a}(2) drops $\Rightarrow$ +1 dim & $4$ \\
\hline
$m=0$ & $n+k=0$, $q\neq 0$ & $d_{31}$ unconstrained; generic & $3$ \\
\hline
$m=0$ & $n+k=0$, $q\neq 0$, $k=-1$ & Also $n=1$ $\Rightarrow$ $(1-n)d_{22}=0$ void & $4$ \\
\hline
$m=0$ & $n+k=0$, $q=0$ & \eqref{eq:R3a}(2) drops; generic & $4$ \\
\hline
$m=n=k=q=0$ & $p\neq -1$ & Free: $d_{21}, d_{23}, d_{31}, d_{32}, d_{33}$ & $5$ \\
\hline
$m=n=k=q=0$ & $p=-1$ & Also $d_{13}$ free & $6$ \\
\hline
\end{tabular}
\end{center}
\normalsize

\section{The Leibniz Algebra $\pazocal{B}ider(\mathcal{D})$}\label{Sec:bider}
 
    Let us now introduce the Leibniz algebra associated to derivations and diderivations over dialgebras. The construction is based on the fact that $\pazocal{D}ider(\mathcal{D})$ is a $\pazocal{D}er(\mathcal{D})$-bimodule when considering the actions $\delta\cdot d = [\delta,d]$ and $d\cdot \delta =[d,\delta]$, where $[\cdot ,\cdot]$ is the canonical bracket of $\mathfrak{gl}(\mathcal{D})$.
 
    In the previous context, if we had an equivariant map $\varphi:\pazocal{D}ider(\mathcal{D})\rightarrow \pazocal{D}er(\mathcal{D})$, then it would be possible to endow $\pazocal{D}ider(\mathcal{D})$ with a Leibniz algebra structure. Given that so far there is not such an equivariant map, then we need to extend the space to a bigger bimodule where an equivariant map can be built. To do that, let us consider the space $\pazocal{D}ider(\mathcal{D})\oplus \pazocal{D}er(\mathcal{D})$ which is a $\pazocal{D}er(\mathcal{D})$-bimodule under the actions $(\delta + d)\cdot d' \coloneqq [\delta, d']+[d,d']$ and $d'\cdot (\delta + d) \coloneqq [d',\delta]+[d',d]$. The canonical projection $\pi: \pazocal{D}ider(\mathcal{D})\oplus \pazocal{D}er(\mathcal{D})\rightarrow \pazocal{D}er(\mathcal{D})$ is the sought equivariant map.

\begin{theorem}\label{thm:Bider_Leib}
    Let $\mathcal{D}$ be a dialgebra, then $(\pazocal{B}ider(\mathcal{D}), \langle \bullet,\bullet\rangle)$ is a Leibniz algebra, where the bracket is defined by 
    \begin{align*}
            \left\langle \bullet ,\bullet\right\rangle: &\ \pazocal{B}ider(\mathcal{D}) \times \pazocal{B}ider(\mathcal{D})  \to\ \pazocal{B}ider(\mathcal{D}) \\
                    &\ ( \tilde{d} \oplus  d, \tilde{d}' \oplus  d') \to \left\langle  \tilde{d} \oplus  d, \tilde{d}' \oplus  d'\right\rangle \coloneqq [ \tilde{d}, d']\oplus[ d, d'].
        \end{align*}
\end{theorem}

\begin{proof}
    Let $\delta_1, \delta_2, \delta_3\in \pazocal{D}ider(\mathcal{D})$ and $d_1, d_2, d_3 \in\pazocal{D}er(\mathcal{D})$, then

        \begin{description}
            \item [i)]
                \begin{align*}
                    \left\langle \left\langle  \delta_1\oplus  d_1, \delta_2\oplus  d_2\right\rangle, \delta_3\oplus  d_3\right\rangle & = \left\langle [ \delta_1, d_2]\oplus[ d_1, d_2], \delta_3\oplus  d_3\right\rangle \\
                    & = [[ \delta_1, d_2], d_3]\oplus[[ d_1, d_2], d_3]
                \end{align*}
            
            \item [ii)]
                \begin{align*}
                    \left\langle \left\langle  \delta_1\oplus  d_1, \delta_3\oplus  d_3\right\rangle, \delta_2\oplus  d_2\right\rangle & = \left\langle [ \delta_1, d_3]\oplus[ d_1, d_3], \delta_2\oplus  d_2\right\rangle\\
                    & = [[ \delta_1, d_3], d_2]\oplus[[ d_1, d_3], d_2]
                \end{align*}
                
            \item[iii)] 
                \begin{align*}
                    \left\langle   \delta_1\oplus  d_1,\left\langle  \delta_2\oplus  d_2, \delta_3\oplus  d_3\right\rangle\right\rangle & = \left\langle  \delta_1\oplus d_1,[ \delta_2,d_3]\oplus[d_2,d_3]\right\rangle \\
                    & = [ \delta_1,[d_2,d_3]]\oplus[d_1,[d_2,d_3]]
                \end{align*}

\end{description}
    In light of the skew-symmetry property of Jacobi's identity, it is clear that  $[[ \delta_1, d_2], d_3]=[[ \delta_1, d_3], d_2]+[ \delta_1,[ d_2, d_3]]$ and $[[ d_1, d_2], d_3]=[[ d_1, d_3], d_2]+[ d_1,[ d_2, d_3]]$. Using this identity in $\textbf{i)}, \textbf{ii)}$ and $\textbf{iii)}$, we have
        \begin{align*}
            \left\langle \left\langle  \tilde{d}_1\oplus  d_1, \tilde{d}_2\oplus  d_2\right\rangle, \tilde{d}_3\oplus  d_3\right\rangle
            & = \left\langle \left\langle  \tilde{d}_1\oplus  d_1, \tilde{d}_3\oplus  d_3\right\rangle, \tilde{d}_2\oplus  d_2\right\rangle+ \\
            & \hspace{0.4 cm} \left\langle   \tilde{d}_1\oplus  d_1,\left\langle  \tilde{d}_2\oplus  d_2, \tilde{d}_3\oplus  d_3\right\rangle\right\rangle
        \end{align*}
    and therefore $(\pazocal{B}ider(\mathcal{D}), \langle \bullet,\bullet\rangle)$ is a Leibniz algebra.
\end {proof}
 
    From Theorems \ref{derL*} and \ref{derR*} we have the following result.

\begin{lemma}
    The subspaces $\pazocal{D}\mathcal{I}nn(\mathcal{D})\oplus \pazocal{D}er(\mathcal{D})$ and $\pazocal{D}\mathcal{I}nn(\mathcal{D})\oplus \mathcal{I}nn(\mathcal{D})$ are ideals of $\pazocal{B}ider(\mathcal{D})$. Moreover $\pazocal{B}ider(\mathcal{D})^{ann}\subseteq \pazocal{D}ider(\mathcal{D})\oplus \{0\}$.
\end{lemma}

\section{Derivations and Diderivations of the Dialgebra $K[x,y]$}\label{Sec:example}
 
    Finally, we present an example of a dialgebra, and we classify both its derivations and diderivations. We consider the dialgebra introduced by L. Lin and Y. Zhang in \cite{Lin2010} generated by polynomials in two commuting variables $K[x,y]$.
 
    Let us start by defining the products $\vdash$ and $\dashv$:
    
        \begin{align*}
            f(x,y)\dashv g(x,y)=f(x,y)g(y,y),
            \hspace{0.3 cm}
            f(x,y)\vdash g(x,y)=f(x,x)g(x,y).
        \end{align*}
    An easy calculation shows that $(K[x,y],\vdash, \dashv)$ is a dialgebra, \cite[Theorem 1.2.1]{Lin2010}. Moreover, it is clear that $\left\{x^my^n\; | \;m,n\in \mathbb{N} \right\}$ is a basis for the dialgebra $K[x,y]$.
 
    Now, since the constant polynomial $e(x,y)=1$ is a bar unit of $K[x,y]$:
        \begin{align*}
            f(x,y)\dashv e(x,y)=f(x,y)e(y,y)=f(x,y)1=f(x,y),\\ \\
            e(x,y)\vdash f(x,y)=e(x,x)f(x,y)=1f(x,y)=f(x,y),
        \end {align*}
    we have that $(K[x,y],\vdash,\dashv)$ is a unital dialgebra, and by Lemma \ref{annzb} we may conclude that
        \begin{align*}
            (K[x,y])^{ann}=Z_B(K[x,y])&=\left\{f(x,y)\;|\; 1\dashv f(x,y)=0 \ \text{and} \ f(x,y)\vdash 1=0 \right\}\\
            &=\left\{f(x,y)\; |\; f(x,x)=f(y,y)=0\right\},
        \end{align*}
    in other words, the ideal $(K[x,y])^{ann}$ is generated by the polynomial $(x-y)$. According to Lemma \ref{annzb} we have
        \begin{equation*}
            \mathcal{H}(\mathcal{D})=\left\{1+(y-x)h(x,y)\; | \; h(x,y)\in K[x,y] \right\}.
        \end{equation*}
    This property was also established by L. Lin and Y. Zhang in \cite[Lemma 1.2.5]{Lin2010}.
 
    Let us recall the results about derivations in $K[x,y]$ proved by L. Lin and Y. Zhang. On the one hand, by definition, we know that $x\dashv x = x\dashv y$, $y\dashv y=y\dashv x$, $x\dashv x=x\dashv y$, $x\vdash 1=y\vdash 1=x$ and $1\dashv x =1\dashv y=y$. Therefore
        \begin{align*}
            x^my^n & = x^m\,1\,y^n \\
            & = x^m\vdash 1 \dashv x^n \\
            & = \underbrace{x\vdash x \vdash \cdots \vdash x}_{m}\vdash 1 \dashv \underbrace{x\dashv x\cdots \dashv x}_{n}.
        \end{align*}
 
    On the other hand, if $d\in \pazocal{D}er(K[x,y])$ then
        \begin{equation*}
            d(x)=d(x\vdash 1)=d(x)\vdash 1+x\vdash d(1)=d(x)\vdash 1+x\,d(1).
        \end{equation*}
    Since $d(1)\in (K[x,y])^{ann}$ (see the argument before \cite[Theorem 2.13]{Lin2010}), there exists $g(x,y)\in K[x,y]$ such that $d(1)=(x-y)g(x,y)$. Considering $d(x)=f(x,y)$, then $f(x,y)=f(x,x)+x[(x-y)g(x,y)] $ and we have
        \begin{align*}
            d(x^my^n)=&d(x\vdash x \vdash \cdots \vdash x\vdash 1 \dashv x\dashv x\cdots\dashv x)\\
            = & (d(x)\vdash x \vdash \cdots \vdash x\vdash 1  \dashv x\dashv x\cdots\dashv x)+ \\
            & \hspace{0.2 cm} \cdots + (x\vdash x \vdash \cdots \vdash x\vdash d(1) \dashv x\dashv x\cdots\dashv x)\\
            &+(x\vdash x \vdash \cdots \vdash x\vdash 1 \dashv d(x)\dashv x\cdots\dashv x)+ \\
            & \hspace{0.2 cm} \cdots + (x\vdash x \vdash \cdots \vdash x\vdash 1 \dashv x\dashv x\cdots\dashv d(x))\\
            &=mx^{m-1}y^nf(x,x)+x^my^nd(1)+nx^my^{n-1}f(y,y).
        \end{align*}
    From this identity L. Lin and Y. Zhang proved the following theorems. 

\begin{theorem}\cite[Theorem 2.3.2]{Lin2010}\label{Char_der_Lin-1}
    Let $d\in \pazocal{D}er(F[x,y])$. There exist $f(x)\in K[x]$ and $g(x,y)\in K[x,y]$, such that
        \begin{equation}\label{DF}
            d(x^my^n)=mx^{m-1}y^nf(x)+x^my^n(x-y)g(x,y)+nx^my^{n-1}f(y).
        \end{equation}
    Conversely, for any $f(x)\in K[x]$ and $g(x,y)\in K[x,y]$, the linear transformation defined by Equation (\ref{DF}) is a derivation of $K[x,y]$.
\end{theorem}

\begin{theorem}\cite[Theorem 2.3.3]{Lin2010}
    Let $d\in \mathcal{I}nn(K[x,y])$, then there exist $h(x)\in K[x]$,  $g(x,y)\in K[x,y]$ and $f(x)=0$, such that
        \begin{equation}\label{DIF}
            (x-y)g(x,y)=h(y)-h(x)
        \end{equation} 
    and its image over the canonical basis of $K[x,y]$ is given by Identity(\ref{DF}). 
 
    Conversely, let $h(x)\in K[x]$, $g(x,y)\in K[x,y]$ and $f(x)=0$, such that (\ref{DIF}) holds, then the linear transformation $d$ generated by (\ref{DF}) is an inner derivation of $K[x,y]$. 
\end{theorem}

\begin{remark}
    It is important to point out that
        \begin{enumerate}
	       \item $d(y)=d(1)\dashv x+ 1\dashv          d(x)$, because of $d(y)=d(1\dashv x)$. 
	       \item In the first implication of Theorem \ref{Char_der_Lin-1} $(x-y)g(x,y)=d(1)$ and $f(x)=d(x)\vdash 1$. This means that every derivation is determined by its action over $d(x)$ and  $d(1)$.
\end{enumerate}
\end{remark}
 
    Let us now characterize diderivations over 
    $K[x,y]$. To simplify the reasoning, let us introduce the following notation:
        \begin{equation*}
            x^m_{\vdash}=\underbrace{x\vdash x\vdash x\cdots \vdash x}_{m\text{-times}}
            \hspace{0.2 cm}
            \text{ and }
            \hspace{0.2 cm}
            y^n_{\dashv}=\underbrace{y\dashv y\dashv y\cdots \dashv y}_{n\text{-times}}
        \end{equation*}
 
    By definition we know that
        \begin{equation*}
            x^m_{\vdash}=y^{m-1}\vdash x
            \hspace{0.2 cm}
            \text{ and }
            \hspace{0.2 cm}
            y^n_{\dashv}=y\dashv x^{n-1}.
        \end{equation*}

\begin{theorem}
Let $\delta\in \pazocal{D}ider(K[x,y])$, then
    \begin{enumerate}
        \item $\delta(x^m)=\displaystyle\sum_{k=0}^{m-1}{(x^k_{\vdash})\vdash \delta(x)\dashv(y^{m-k-1}_{\dashv})}=\delta(x)\displaystyle\sum_{k=0}^{m-1}{x^ky^{m-k-1}}$.

        \item $\delta(y^n)=\displaystyle\sum_{k=0}^{n-1}{(x^k_{\vdash})\vdash \delta(y)\dashv(y^{n-k-1}_{\dashv})}=\delta(y)\displaystyle\sum_{k=0}^{n-1}{x^ky^{n-k-1}}$.
        
        \item $\delta(x^my^n)=\delta(x)y^n\displaystyle\sum_{k=0}^{m-1}{x^ky^{m-k-1}}+ \delta(y)x^m\displaystyle\sum_{k=0}^{n-1}{x^ky^{n-k-1}}$
    \end{enumerate} 
\end{theorem}

\begin{proof}

\begin{enumerate}
	\item To show $\delta(x^m)=\delta(x)\sum_{k=0}^{m-1}{x^ky^{m-k-1}}$ we proceed by induction.  If
	$m=2$,
	   \begin{align*}
	       \delta(x^2)&=\delta(x\vdash   x)=\delta(x)\dashv x+x\vdash\delta(x)\\
	   &=\delta(x)y+x\delta(x)=\delta(x)      (y+x)=\delta                         (x)\displaystyle\sum_{k=0}^{1}{x^ky^{1-k}}.
	   \end{align*}
    Let us assume the result for $m\in\mathbb{N}_{> 2}$ and let us show that it also holds for $m+1$. We have, 
	   \begin{align*}
	       \delta(x^{m+1})&=\delta(x^m\vdash x)=\delta(x^m)\dashv x+x^m\vdash\delta(x)\\
	       &=\left(\delta(x)\displaystyle\sum_{k=0}^{m-1}{x^ky^{m-k-1}}\right)y+x^m\delta(x)\\
	       &=\delta(x)\displaystyle\sum_{k=0}^{m-1}{x^ky^{m-k}}+x^m\delta(x)\\
	       &=\delta(x)\left( \displaystyle\sum_{k=0}^{m-1}{x^ky^{m-k}}+x^m\right)\\
	       &=\delta(x)\displaystyle\sum_{k=0}^{m}{x^ky^{m-k}},
	   \end{align*}
	\item The proof of this item follows the same lines of reasoning as in 1.
	\item All in all, we have
		\begin{align*}
	   \delta(x^my^n)&=\delta(x^m)\dashv y^n+x^m\vdash\delta(y^n)\\
	   &=\left(\delta(x)\displaystyle\sum_{k=0}^{m-1}{x^ky^{m-k-1}}\right)\dashv y^n+x^m\vdash \left(\delta(y)\displaystyle\sum_{k=0}^{n-1}{x^ky^{n-k-1}}\right)\\
        &=\delta(x)y^n\displaystyle\sum_{k=0}^{m-1}{x^ky^{m-k-1}}+\delta(y)x^m\displaystyle\sum_{k=0}^{n-1}{x^ky^{n-k-1}}
		\end{align*}
    \end{enumerate}
\end{proof}
 
    Let us finally consider the identity
        \begin{equation*}
            (x-y)\left(\displaystyle\sum_{k=0}^{n-1}{x^ky^{n-k-1}}\right)=x^n-y^n,
        \end{equation*}
    From the third part in the previous theorem 
        \begin{equation*}
            \delta(x^my^n)=\delta(x)\, y^n\, \left( \displaystyle\frac{x^m-y^m}{x-y} \right) + \delta(y)\, x^m  \left( \displaystyle \frac{x^n-y^n}{x-y} \right),
        \end{equation*}
 
    which shows the following characterization.

\begin{lemma}
    If $\delta\in \pazocal{D}ider(K[x,y])$, then
        \begin{equation*}
            \delta(h(x,y))=\frac{1}{x-y} \left[ \delta(x) \Big( h(x,y)-h(y,y)  \Big) - \delta(y) \Big( h(x,y)-h(x,x)  \Big) \right] 
        \end{equation*}
    for every $h(x,y)\in K[x,y]$.
\end{lemma}

\begin{theorem}\label{thm:dider_identity}
    Let $f(x,y),\; g(x,y)\in K[x,y]$, then the linear transformation $\delta$ of $K[x,y]$ induced by the identity
        \begin{equation*}
            \delta(x^my^n)=f(x,y)\, y^n \left( \displaystyle\frac{x^m-y^m}{x-y} \right) + g(x,y)\, x^m \left( \displaystyle \frac{x^n-y^n}{x-y} \right),
        \end{equation*}
    is a diderivation of $K[x,y]$. Furthermore, in this case $\delta(x)=f(x,y)$ and $\delta(y)=g(x,y)$.
\end{theorem}
 
    Finally, like for diderivations, we have:

\begin{lemma}\label{lem:inner_dider}
    For every $p(x,y)\in K[x,y]$, 
        \begin{equation}\label{fin}
            Ad_{p(x,y)}(h(x,y))=p(x,y)\Big( h(x,x)-h(y,y) \Big),
        \end{equation}
    for all $h(x,y)\in K[x,y]$. Therefore $Ad_{p(x,y)} \big( K[x,y] \big) \subseteq (K[x,y])^{ann}$.
 
    Conversely, if $\delta$ is a linear transformation satisfying  identity (\ref{fin}), then $\delta =Ad_{p(x,y)}$. 
\end{lemma}

\section{Conclusion and Further Work}\label{Sec:conclusions}

In this article we have investigated derivations and biderivations in the framework of dialgebras, introducing the new concept of diderivations and analyzing their interaction with derivations through multiplicative operators. A further central result is the construction of the algebra of biderivations and the proof that it naturally carries a Leibniz algebra structure, thereby extending the classical correspondence between associative, Lie, and Leibniz settings. In addition, we have presented a complete classification of the vector spaces of diderivations for dialgebras of dimensions two and three. These low-dimensional computations not only corroborate the general framework but also reveal structural patterns that inform the extension of the theory. The article also completes the classification of operators on the polynomial dialgebra $F[x,y]$, concluding with an explicit description of its diderivations.

Future research directions include a systematic study of cohomology theories of dialgebras with coefficients and their interplay with tensor and operadic categories. Another promising avenue is the relationship between derivations and deformation theory, which may yield new insights with potential applications to mathematical physics, particularly in connection with open problems surrounding the Yang–Baxter equation.

\begin{appendices}

\section{Python Code}\label{Latex_code}
 
    We have used the following code to solve system (\ref{System_dider}) when $\mathcal{D}=Dias_2^1$.

\begin{lstlisting}
import sympy as sp

# Define symbolic variables d11, d12, d21, d22
d11, d12, d21, d22 = sp.symbols('d11 d12 d21 d22')

# Define symbolic matrix D
D = sp.Matrix([[d11, d12], [d21, d22]])

# Define matrices for L and R (grouped into dictionaries for easy access)
L_matrices = {
    'L1pb': sp.Matrix([[1, 0], [0, 0]]),
    'L2pb': sp.Matrix([[0, 0], [1, 0]]),
    'L1bp': sp.Matrix([[1, 0], [0, 0]]),
    'L2bp': sp.Matrix([[0, 0], [0, 0]])
}

R_matrices = {
    'R1pb': sp.Matrix([[1, 0], [0, 1]]),
    'R2pb': sp.Matrix([[0, 0], [0, 0]]),
    'R1bp': sp.Matrix([[1, 0], [0, 0]]),
    'R2bp': sp.Matrix([[0, 0], [0, 0]])
}
# Compute compositions of D with L and R matrices
compositions = {}
for matrix_name, matrix in {**L_matrices, **R_matrices}.items():
    compositions[f"D_{matrix_name}"] = D * matrix
    compositions[f"{matrix_name}_D"] = matrix * D
# Define the right-hand side (RHS) of each equation
rhs_results = {
    'L1bp': d11 * L_matrices['L1pb'] + d21 * L_matrices['L2pb'],
    'L2bp': d12 * L_matrices['L1pb'] + d22 * L_matrices['L2pb'],
    'R1pb': d11 * R_matrices['R1bp'] + d21 * R_matrices['R2bp'],
    'R2pb': d12 * R_matrices['R1bp'] + d22 * R_matrices['R2bp']
}
\end{lstlisting}
\newpage
\begin{lstlisting}
# Solve the equations
solutions = {}
for matrix_type in ['L1bp', 'L2bp', 'R1pb', 'R2pb']:
    lhs_result = compositions[f"D_{matrix_type}"] - compositions[f"{matrix_type}_D"]
    rhs_result = rhs_results[matrix_type]
    solutions[matrix_type] = sp.solve(lhs_result - rhs_result)
# Print the solutions for each matrix type
for matrix_type, solution in solutions.items():
    print(f"Solution for {matrix_type}: {solution}")
\end{lstlisting}
 
    The code involves symbolic matrices depending on the dialgebra $Dias_{i}^{j}$. Furthermore, the approach employs symbolic variables to define matrices and solve the system algebraically. Let us explain its structure.
 
    We start defining the symbolic variables $d11$, \( d12 \), \( d21 \), and \( d22 \).

\begin{lstlisting}[language=Python]
import sympy as sp

d11, d12, d21, d22 = sp.symbols('d11 d12 d21 d22')
\end{lstlisting}
 
     These variables will be used to construct matrices and represent unknown values in the system of equations, so that we can manipulate them algebraically during computation.

\begin{lstlisting}[language=Python]
D = sp.Matrix([[d11, d12], [d21, d22]])
\end{lstlisting}
     
    Now, two dictionaries, \texttt{L\_matrices} and \texttt{R\_matrices}, are defined to store the matrices associated with the multiplicative operators $L_{e_i}^{\vdash}, L_{e_i}^{\dashv}, R_{e_{i}}^{\vdash}$ and $R_{e_i}^{\dashv}$. For instance, the matrix associated to $Dias_{2}^1$ are
    
\begin{lstlisting}[language=Python]
L_matrices = {
    'L1pb': sp.Matrix([[1, 0], [0, 0]]),
    'L2pb': sp.Matrix([[0, 0], [1, 0]]),
    'L1bp': sp.Matrix([[1, 0], [0, 0]]),
    'L2bp': sp.Matrix([[0, 0], [0, 0]])

}

R_matrices = {
    'R1pb': sp.Matrix([[1, 0], [0, 1]]),
    'R2pb': sp.Matrix([[0, 0], [0, 0]]),
    'R1bp': sp.Matrix([[1, 0], [0, 0]]),
    'R2bp': sp.Matrix([[0, 0], [0, 0]])
}
\end{lstlisting}
 
    \texttt{L\_matrices}: represents the left operators: \texttt{Lipb}=$L_{e_i}^{\dashv}$ and \texttt{Libp}=$L_{e_i}^{\vdash}$, and
 
    \texttt{R\_matrices}: represents the right operators: \texttt{Ripb}=$R_{e_i}^{\dashv}$ and \texttt{Ribp}=$R_{e_i}^{\vdash}$.
 
     For each matrix in both the \texttt{L\_matrices} and \texttt{R\_matrices} dictionaries, the code

\begin{lstlisting}[language=Python]
compositions = {}
for matrix_name, matrix in {**L_matrices, **R_matrices}.items():
    compositions[f"D_{matrix_name}"] = D * matrix
    compositions[f"{matrix_name}_D"] = matrix * D
\end{lstlisting}
 
    computes two compositions:

\begin{itemize}
    \item \( D \times M \) (for each matrix $M$ in \texttt{L\_matrices} and \texttt{R\_matrices}).
    \item \( M \times D \) (for each $M$ matrix in \texttt{L\_matrices} and \texttt{R\_matrices}),
\end{itemize}
 
    which are stored in the dictionary \texttt{compositions} with the following keys: \( D\_L1pb \), \( L1pb\_D \), \( D\_L2pb \), \( L2pb\_D \), \( D\_R1pb \), \( R1pb\_D \), \( D\_R2pb \) and \( R2pb\_D \).
 
    For each matrix type (\( L1bp \), \( L2bp \), \( R1pb \), \( R2pb \)), the RHS of (\ref{System_dider}) consists of a linear combination of matrices (such as \( L1pb \), \( L2pb \), etc.) multiplied by the symbolic variables \( d11 \), \( d12 \), \( d21 \), and \( d22 \). These expressions will be used in the subsequent step to form the full equation to be solved.

\begin{lstlisting}[language=Python]
rhs_results = {
    'L1bp': d11 * L_matrices['L1pb'] + d21 * L_matrices['L2pb'],
    'L2bp': d12 * L_matrices['L1pb'] + d22 * L_matrices['L2pb'],
    'R1pb': d11 * R_matrices['R1bp'] + d21 * R_matrices['R2bp'],
    'R2pb': d12 * R_matrices['R1bp'] + d22 * R_matrices['R2bp']
}
\end{lstlisting}
 
     For each matrix type (\( L1bp \), \( L2bp \), \( R1pb \), \( R2pb \)) we have defined the following expressions:

\begin{itemize}
    \item \textbf{Left-hand side (LHS)}: The composition \( D \times M - M \times D \) is calculated from the \texttt{compositions} dictionary.
    \item \textbf{Right-hand side (RHS)}: The corresponding RHS expression from the \texttt{rhs\_results} dictionary.
\end{itemize}
 
    The difference between LHS and RHS is passed to SymPy’s \texttt{solve()} function to find the solutions to the system. The results are stored in the \texttt{solutions} dictionary.

\begin{lstlisting}[language=Python]
solutions = {}
for matrix_type in ['L1bp', 'L2bp', 'R1pb', 'R2pb']:
    lhs_result = compositions[f"D_{matrix_type}"] - compositions[f"{matrix_type}_D"]
    rhs_result = rhs_results[matrix_type]
    solutions[matrix_type] = sp.solve(lhs_result - rhs_result)
\end{lstlisting}
 
    Finally, we iterate over the \texttt{solutions} dictionary and print the solutions for each matrix type. 
    
\begin{lstlisting}[language=Python]
for matrix_type, solution in solutions.items():
    print(f"Solution for {matrix_type}: {solution}")
\end{lstlisting}

\section{Three-Dimensional Dialgebras}
 
    From \cite{Abubakar2014} we have the following classification of the isomorphism classes of three dimensional dialgebras.

\begin{table}[ht]
\centering
\renewcommand{\arraystretch}{1.2} % Aumenta el espaciado entre filas
\setlength{\tabcolsep}{4pt} % Ajusta el espaciado entre columnas
\small % Ajusta el tamaño de la fuente
\begin{tabular}{|p{7cm}|p{7cm}|}
\hline
\textbf{$Dias_{3}^1$} & \textbf{$Dias_{3}^2$} \\
\hline
$e_1 \vdash e_2 = e_1$, $e_2 \vdash e_2 = e_2$, $e_3 \vdash e_3 = e_3$ & \\
$e_2 \dashv e_2 = e_2$, $e_3 \dashv e_3 = e_3$
&
$e_1 \vdash e_2 = e_1$, $e_2 \vdash e_2 = e_2$, $e_3 \vdash e_3 = e_3$ \\
$e_2 \dashv e_1 = e_1$, $e_2 \dashv e_2 = e_2$, $e_3 \dashv e_3 = e_3$ &\\
\hline
\textbf{$Dias_{3}^3$} & \textbf{$Dias_{3}^4$} \\
\hline
$e_1 \vdash e_2 = e_1$, $e_2 \vdash e_2 = e_2$, $e_3 \vdash e_3 = e_3$ & \\
$e_2 \dashv e_2 = e_2$, $e_3 \vdash e_1 = e_1$, $e_3 \dashv e_3 = e_3$
&
$e_1 \vdash e_3 = e_2$, $e_2 \vdash e_3 = e_2$, $e_3 \vdash e_3 = e_3$ \\
$e_3 \dashv e_3 = e_3$ & \\
\hline
\textbf{$Dias_{3}^5$} & \textbf{$Dias_{3}^6$} \\
\hline
$e_1 \vdash e_3 = e_2$, $e_2 \vdash e_3 = e_2$, $e_3 \vdash e_3 = e_3$ & \\
$e_3 \vdash e_1 = e_1 - e_2$, $e_3 \dashv e_3 = e_3$
&
$e_1 \vdash e_3 = e_2$, $e_2 \vdash e_3 = e_2$, $e_3 \vdash e_3 = e_3$ \\
$e_3 \vdash e_1 = e_1$, $e_3 \vdash e_2 = e_2$, $e_3 \dashv e_3 = e_3$ & \\
\hline
\textbf{$Dias_{3}^7$} & \textbf{$Dias_{3}^8$} \\
\hline
$e_1 \vdash e_3 = e_2$, $e_2 \vdash e_3 = e_2$, $e_3 \vdash e_3 = e_3$ & \\
$e_3 \vdash e_1 = e_2$, $e_3 \vdash e_2 = e_2$, $e_3 \dashv e_3 = e_3$
&
$e_1 \vdash e_3 = e_2$, $e_2 \vdash e_3 = e_2$, $e_3 \vdash e_3 = e_3$ \\
$e_1 \vdash e_3 = e_2$, $e_3 \vdash e_1 = e_1 - e_2$, $e_3 \dashv e_3 = e_3$ & \\
\hline
\textbf{$Dias_{3}^9$} & \textbf{$Dias_{3}^{10}$} \\
\hline
$e_3 \vdash e_1 = e_1$, $e_3 \vdash e_2 = e_2$, $e_3 \vdash e_3 = e_3$
&
$e_3 \vdash e_1 = e_1$, $e_3 \vdash e_3 = e_3$ \\
\hline
\textbf{$Dias_{3}^{11}$} & \textbf{$Dias_{3}^{12}$} \\
\hline
$e_3 \vdash e_1 = e_1$, $e_3 \vdash e_2 = e_2$, $e_3 \vdash e_3 = e_3$
&
$e_1 \vdash e_3 = e_1$, $e_2 \vdash e_3 = e_2$, $e_3 \vdash e_1 = e_1$, $e_3 \vdash e_3 = e_3$ \\
\hline
\textbf{$Dias_{3}^{13}$} & \textbf{$Dias_{3}^{14}$} \\
\hline
$e_1 \vdash e_3 = e_1$, $e_2 \vdash e_3 = e_2$, $e_3 \vdash e_1 = e_1$, $e_3 \vdash e_2 = e_2$, $e_3 \vdash e_3 = e_3$
&
$e_1 \vdash e_3 = e_1 + e_2$, $e_3 \vdash e_1 = e_1$, $e_3 \vdash e_2 = e_2$, $e_3 \vdash e_3 = e_3$ \\
\hline
\textbf{$Dias_{3}^{15}$} & \textbf{$Dias_{3}^{16}$} \\
\hline
$e_1 \vdash e_1 = e_2$, $e_3 \vdash e_3 = e_3$
&
$e_{1}\dashv e_{3}=e_{2}, 
e_{3}\vdash e_{1}=k e_{2}, 
e_{1}\dashv e_{1}=m e_{2}, 
e_{1}\vdash e_{3}=n e_{2}, 
e_{3}\dashv e_{1}=p e_{2}, 
e_{3}\vdash e_{3}=q e_{2}.$ \\
\hline
\textbf{$Dias_{3}^{17}$} & \\ % Última celda vacía para cerrar la tabla correctamente
\hline
$e_{1}\dashv e_{3}=e_{2},
e_{3}\vdash e_{1}=l e_{2},
e_{1}\dashv e_{1}=m e_{2},
e_{1}\vdash e_{3}=n e_{2}, 
e_{3}\dashv e_{1}=p e_{2}, 
e_{3}\vdash e_{3}=q e_{2}$
& \\
\hline
\end{tabular}
\caption{Three-dimensional complex dialgebras.}
\label{Three-dimensional-dias}
\end{table}

\section{Diderivations over $Dias_3^{16}$}\label{App:Dias_3^16}

Let
\[
D=\begin{pmatrix}
d_{11}&d_{12}&d_{13}\\
d_{21}&d_{22}&d_{23}\\
d_{31}&d_{32}&d_{33}
\end{pmatrix}.
\]
and let us fix scalar parameters: $k,m,n,p,q\in\mathbb{C}$, which come from the previous table.

The left operators:
\[
\begin{aligned}
&L_{1\mathrm{pb}}=\begin{pmatrix}0&0&0\\[2pt]0&0&1\\[2pt]0&0&0\end{pmatrix},\quad
L_{2\mathrm{pb}}=0,\quad
L_{3\mathrm{pb}}=\begin{pmatrix}0&0&0\\[2pt]k&0&0\\[2pt]0&0&0\end{pmatrix},\\[6pt]
&L_{1\mathrm{bp}}=\begin{pmatrix}0&0&0\\[2pt]m&0&n\\[2pt]0&0&0\end{pmatrix},\quad
L_{2\mathrm{bp}}=0,\quad
L_{3\mathrm{bp}}=\begin{pmatrix}0&0&0\\[2pt]p&0&q\\[2pt]0&0&0\end{pmatrix},
\end{aligned}
\]

and the right operators:
\[
\begin{aligned}
&R_{1\mathrm{pb}}=\begin{pmatrix}0&0&0\\[2pt]0&0&k\\[2pt]0&0&0\end{pmatrix},\quad
R_{2\mathrm{pb}}=0,\quad
R_{3\mathrm{pb}}=\begin{pmatrix}0&0&0\\[2pt]1&0&0\\[2pt]0&0&0\end{pmatrix},\\[6pt]
&R_{1\mathrm{bp}}=\begin{pmatrix}0&0&0\\[2pt]m&0&p\\[2pt]0&0&0\end{pmatrix},\quad
R_{2\mathrm{bp}}=0,\quad
R_{3\mathrm{bp}}=\begin{pmatrix}0&0&0\\[2pt]n&0&q\\[2pt]0&0&0\end{pmatrix},
\end{aligned}
\]

together with eqution (\ref{System_dider}) give rise to the six equalities:
\begin{align}
D L_{1\mathrm{bp}}-L_{1\mathrm{bp}}D&=d_{11}L_{1\mathrm{pb}}+d_{21}L_{2\mathrm{pb}}+d_{31}L_{3\mathrm{pb}},\label{eq:L1}\\
D L_{2\mathrm{bp}}-L_{2\mathrm{bp}}D&=d_{12}L_{1\mathrm{pb}}+d_{22}L_{2\mathrm{pb}}+d_{32}L_{3\mathrm{pb}},\label{eq:L2}\\
D L_{3\mathrm{bp}}-L_{3\mathrm{bp}}D&=d_{13}L_{1\mathrm{pb}}+d_{23}L_{2\mathrm{pb}}+d_{33}L_{3\mathrm{pb}},\label{eq:L3}\\
D R_{1\mathrm{pb}}-R_{1\mathrm{pb}}D&=d_{11}R_{1\mathrm{bp}}+d_{21}R_{2\mathrm{bp}}+d_{31}R_{3\mathrm{bp}},\label{eq:R1}\\
D R_{2\mathrm{pb}}-R_{2\mathrm{pb}}D&=d_{12}R_{1\mathrm{bp}}+d_{22}R_{2\mathrm{bp}}+d_{32}R_{3\mathrm{bp}},\label{eq:R2}\\
D R_{3\mathrm{pb}}-R_{3\mathrm{pb}}D&=d_{13}R_{1\mathrm{bp}}+d_{23}R_{2\mathrm{bp}}+d_{33}R_{3\mathrm{bp}}.\label{eq:R3}
\end{align}

\begin{remark}
Since $L_{2\mathrm{pb}}=R_{2\mathrm{pb}}=0$, the coefficients $d_{21},d_{23}$ never occur. In other words, they are always free.
\end{remark}

An independent set of scalar equations extracted from \eqref{eq:L1}--\eqref{eq:R3} is:
\begin{align}
&\textbf{(i) From \eqref{eq:L1}:}\quad
m\,d_{12}=0,\;\; n\,d_{12}=0,\;\; m\,d_{32}=0,\;\; n\,d_{32}=0, \label{eq:L1a}\\
&\hspace{1cm}(-d_{11}+d_{22})\,m - d_{31}\,n = d_{31}\,k, \label{eq:L1b}\\
&\hspace{1cm}-d_{13}\,m + (d_{22}-d_{33})\,n = d_{11}. \label{eq:L1c}
\\[4pt]
&\textbf{(ii) From \eqref{eq:L2}:}\quad k\,d_{32}=0,\qquad d_{12}=0. \label{eq:L2a}
\\[4pt]
&\textbf{(iii) From \eqref{eq:L3}:}\quad
(-d_{11}+d_{22})\,p - d_{31}\,q = d_{33}\,k,\label{eq:L3a}\\
&\hspace{1cm}-d_{13}\,p + (d_{22}-d_{33})\,q = d_{13}. \label{eq:L3b}
\\[4pt]
&\textbf{(iv) From \eqref{eq:R1}:}\quad
-d_{31}\,k = d_{11}\,m + d_{31}\,n, \label{eq:R1a}\\
&\hspace{1cm}k\,(d_{22}-d_{33}) = d_{11}\,p + d_{31}\,q. \label{eq:R1b}
\\[4pt]
&\textbf{(v) From \eqref{eq:R2}:}\quad d_{12}\,m + d_{32}\,n = 0,\qquad d_{12}\,p + d_{32}\,q=0.\label{eq:R2a}
\\[4pt]
&\textbf{(vi) From \eqref{eq:R3}:}\quad
-d_{11}+d_{22}= d_{13}\,m + d_{33}\,n,\hspace{0.2 cm}
(p+1)\,d_{13} + q\,d_{33}=0. \label{eq:R3a}
\end{align}

\subsubsection*{Initial reduction}
From Equations \eqref{eq:L2a} and \eqref{eq:R3a}: $d_{12}=0$ and, except in the extreme $m=n=k=q=0$, also $d_{32}=0$. If $m\neq 0$, then combining \eqref{eq:L1b} and \eqref{eq:R1a} yields $d_{22}=0$.

We work with the $5\times5$ linear subsystem in the unknowns

\begin{equation*}
    \mathbf x=(d_{11},d_{13},d_{22},d_{31},d_{33})^\top,
\end{equation*}
extracted from (R1a), (R1b), (L3a), (R3a.2), (R3a.1):

\begin{equation}\label{F.1}
M=
\begin{pmatrix}
m & 0 & 0 & (n+k) & 0\\
-p & 0 & -k & -q & k\\
-p & 0 & p & -q & -k\\
0 & (p+1) & 0 & 0 & q\\
-1 & -m & 1 & 0 & -n
\end{pmatrix}.
\end{equation}

Let us define the two key factors
\begin{equation}\label{F.2}
\boxed{\ \Delta_1:=-k-mq+np\ },\qquad
\boxed{\ \Delta_2:=(k+n)(p+1)-mq\ }.
\end{equation}

Row 4 is particularly simple:
\begin{equation} \label{A.2}
    (p+1)\,d_{13} + q\,d_{33}=0.
\end{equation}

Permuting rows/columns so that the $d_{13}$ equation (row 4) and the
$d_{13}$ variable (column 2) come first. Then $M$ has the following block form

\begin{equation}\label{A.3}
M=\begin{pmatrix}
A & B\\ C & D
\end{pmatrix},
\qquad
A=(p+1)\in \mathbb{C}.
\end{equation}

Here $A$ is $1\times 1$, $B$ is $1\times 4$, $C$ is $4\times 1$, and
$D$ is $4\times 4$ (explicit expressions are immediate from (\ref{F.1})).

\subsubsection*{Schur complement}

Taking into account the decomposition (\ref{A.3}), it is well-known that if $A$ is invertible, then
\begin{equation}\label{eq:det-M}
    \det M = \det A\cdot \det(D - C A^{-1} B).
\end{equation}

Here $A=(p+1)$, so provided $p\neq -1$ we may form the
Schur complement
\[
\widetilde A = D - \tfrac{1}{p+1}\, C B.
\]

Computing $CB$ explicitly gives
\[
CB=
\begin{pmatrix}
0&0&0&0\\
0&0&0&0\\
0&0&0&0\\
0&0&0&-m q
\end{pmatrix},
\]
so that
\[
\widetilde A=
\begin{pmatrix}
m & 0 & (n+k) & 0\\
-p & -k & -q & k\\
-p & \;\;p & -q & -k\\
-1 & \;\;1 & 0 & -n + \tfrac{m q}{p+1}
\end{pmatrix}.
\]

Multiplying the last row by $(p+1)$ yields the cleaner form
\begin{equation}\label{F.4}
\widetilde A=
\begin{pmatrix}
m & 0 & (n+k) & 0\\
-p & -k & -q & k\\
-p & \;\;p & -q & -k\\
-(p+1) & (p+1) & 0 & m q - n(p+1)
\end{pmatrix}.
\end{equation}

Thus
\begin{equation}\label{A.8}
\det M \doteq \det \widetilde A,
\end{equation}
where $\doteq$ means equality up to a nonzero unit.

Apply the row operations
\[
R_2\!\leftarrow\! mR_2+pR_1,\qquad
R_3\!\leftarrow\! mR_3+pR_1,\qquad
R_4\!\leftarrow\! mR_4+(p+1)R_1,
\]
we obtain
\[
\widetilde A'=
\begin{pmatrix}
m & 0 & (n+k) & 0\\
0 & -km & -mq+(k+n)p & km\\
0 & mp  & -mq+(k+n)p & -km\\
0 & m(p+1) & -mq+(k+n)(p+1) & m^2 q
\end{pmatrix}.
\]
Hence the first column has pivot \(m\) and zeros below, so
\begin{equation}\label{F.8}
\det\widetilde A
= m\cdot \det C,\hspace{0.2 cm}
C=
\begin{pmatrix}
-km & -mq+(k+n)p & km\\
mp  & -mq+(k+n)p & -km\\
m(p+1) & -mq+(k+n)(p+1) & m^2 q
\end{pmatrix}.
\end{equation}

Let \(S:=-mq+(k+n)p\). Then the middle column of \(C\) is \((S,\ S,\ S+(k+n))^\top\).
Using \(R_2\leftarrow R_2-R_1\), \(R_3\leftarrow R_3-R_1\), and expanding cofactors along the zero
at $C_{22}$ yields
\begin{align*}
\det C
&=(mp+km)\!\left(-\det\!\begin{pmatrix}S&km\\ k+n&m^2q-km\end{pmatrix}\right)\\
& \hspace{0.4 cm} -2km\!\left(-\det\!\begin{pmatrix}-km&S\\ m(p+1)+km&k+n\end{pmatrix}\right).
\end{align*}

Evaluating the \(2\times2\) minors and simplifying with \(S=-mq+(k+n)p\) one finds
\begin{equation}\label{F.10}
\boxed{\ \det C \doteq \Delta_2\cdot \Delta_1\ }.
\end{equation}

Combining (\ref{eq:det-M}), \eqref{F.8} and \eqref{F.10}:
\[
\boxed{\ \det M \doteq m\,\Delta_2\,\Delta_1\ }.
\]

\begin{remark}
The vanishing locus of the determinant, and hence the rank drops giving rise to extra solution
dimensions, is precisely
\[
\{\,m=0\,\}\ \cup\ \{\,\Delta_2=0\,\}\ \cup\ \{\,\Delta_1=0\,\}.
\]
\end{remark}

\subsection*{Resolution with $m\neq 0$}

\subsubsection*{Case A: trivial solution (generic)}
If
\[
\Delta_1\neq 0\quad\text{and}\quad \Delta_2\neq 0,
\]
then the only possibility is
\[
d_{11}=d_{12}=d_{13}=d_{22}=d_{31}=d_{32}=d_{33}=0,
\]
with \(d_{21},d_{23}\) always free.

\subsubsection*{Case B: non-trivial family when \(\Delta_1=0\)}
Assume
\[
\Delta_1=-k-mq+np=0,\qquad m\neq 0.
\]
From \eqref{eq:R1a}: \(d_{11}= -\dfrac{n+k}{m}\,d_{31}\).
From \eqref{eq:R1b} and \(\Delta_1=0\):
\(d_{33}=\dfrac{p+1}{m}\,d_{31}\).
Using \((p+1)d_{13}+q d_{33}=0\):
\(d_{13}= -\dfrac{q}{m}\,d_{31}\).
With parameter \(t:=d_{33}\):
\[
\boxed{\
(d_{11},d_{12},d_{13},d_{22},d_{31},d_{32},d_{33})
= t\!\left(-\frac{k+n}{p+1},\,0,\,\frac{k-np}{m(p+1)},\,0,\,\frac{m}{p+1},\,0,\,1\right).
}
\]

\subsubsection*{Case C: non-trivial family when \(\Delta_2=0\)}
Assume
\[
\Delta_2=(k+n)(p+1)-m q=0,\qquad m\neq 0.
\]
From \eqref{eq:R1a}: \(d_{11}= -\dfrac{n+k}{m}\,d_{31}\).
From \eqref{eq:R1b} and \(\Delta_2=0\): \(d_{33}= -\dfrac{k+n}{m k}\,d_{31}\).
Using \((p+1)d_{13}+q d_{33}=0\) and \(\Delta_2=0\): \(d_{13}= -\dfrac{k+n}{m}\,d_{33}\).
With parameter \(t:=d_{33}\):
\[
\boxed{\
(d_{11},d_{12},d_{13},d_{22},d_{31},d_{32},d_{33})
= t\!\left(k,\,0,\,-\frac{k+n}{m},\,0,\,-\frac{k\,m}{k+n},\,0,\,1\right),
}
\]
assuming \(k+n\neq 0\).

\subsubsection*{Case D: $m\neq 0$ and $\Delta_1=\Delta_2=0$}

Necessarily
\[
n=-k(p+2),\qquad q=-\frac{k(p+1)^2}{m},\qquad k+n=-k(p+1),
\]
and, from the initial reduction for $m\neq 0$,
\[
d_{12}=d_{22}=d_{32}=0.
\]
The solution set is
\[
(d_{11},d_{13},d_{31},d_{33})
=\Big(\tfrac{k(p+1)}{m}\,d_{31},\ \tfrac{k(p+1)}{m}\,d_{33},\ d_{31},\ d_{33}\Big),
\]
\subsubsection*{Degenerate branches ($m=0$)}
When $m=0$ we do \emph{not} force $d_{22}=0$. The governing block is
\[
(n+k)d_{31}=0,\quad -d_{11}+d_{22}=n\,d_{33},\quad
k(d_{22}-d_{33})=d_{11}p+d_{31}q,
\]
\[
p(-d_{11}+d_{22})-d_{31}q=d_{33}k,
\]
together with \eqref{eq:R3a}(2) and \eqref{eq:L3b}. Also $d_{12}=0$ and
$d_{32}$ is forced to $0$ unless $n=k=q=0$.

\begin{itemize}[leftmargin=1.5em]
\item \textbf{$m=0$, $n+k\neq 0$}: then $d_{31}=0$.
\begin{itemize}
\item[-] If $k\neq n p$ (generic in this branch): $\dim=2$ when $q\neq 0$; if $q=0$ still $\dim=2$ generically.
\item[-] If $k=n p$: one extra free parameter appears in $(d_{11},d_{33})$; hence $\dim=3$. If moreover $p=-1$ and $q=0$ (so \eqref{eq:R3a}(2) vanishes), add $+1$: $\dim=4$.
\end{itemize}

\item \textbf{$m=0$, $n+k=0$}: then $d_{31}$ is not constrained by \eqref{eq:R1a}.
\begin{itemize}
\item[-] For $q\neq 0$ (so $d_{32}=0$): generically $\dim=3$ (one free parameter in $(d_{11},d_{13},d_{22},d_{33})$ plus $d_{21},d_{23}$). If additionally $k=-1$ (i.e., $n=1$), the relation $(1-n)d_{22}=0$ becomes void and gives one more degree: $\dim=4$.
\item[-] For $q=0$: \eqref{eq:R3a}(2) drops; generically this adds $+1$ vs.\ the previous line, so $\dim=4$ (and can rise on further coincidences).
\end{itemize}

\item \textbf{Extreme $m=n=k=q=0$}:
Here $d_{12}=0$, $d_{32}$ is free, and from \eqref{eq:L1c} one gets $d_{11}=0$, then from \eqref{eq:R3a}(1) $d_{22}=0$.
\begin{itemize}
\item[-] If $p\neq -1$: the free directions are $d_{21}, d_{23}, d_{31}, d_{32}, d_{33}$, so $\dim=5$.
\item[-] If $p=-1$: $(p+1)d_{13}=0$ vanishes and \eqref{eq:L3b} becomes tautological, so $d_{13}$ is also free: $\dim=6$.
\end{itemize}
\end{itemize}

\end{appendices}

% ======== BIBLIOGRAPHY ========

\bigskip
\noindent\textsc{Gabriel Gustavo Restrepo-S\'anchez}\\
Instituto de Matem\'aticas, Universidad de Antioquia, Medell\'in, Colombia.\\
\textit{Email address}: \texttt{ggustavo.restrepo@udea.edu.co}

\medskip
\noindent\textsc{Jos\'e Gregorio Rodr\'iguez-Nieto}\\
Departamento de Matem\'aticas, Universidad Nacional de Colombia, Medell\'in, Colombia.\\
\textit{Email address}: \texttt{jgrodrig@unal.edu.co}

\medskip
\noindent\textsc{Olga Patricia Salazar-D\'iaz}\\
Departamento de Matem\'aticas, Universidad Nacional de Colombia, Medell\'in, Colombia.\\
\textit{Email address}: \texttt{opsalazard@unal.edu.co}

\medskip
\noindent\textsc{Andr\'es Sarrazola-Alzate}\\
Departamento de Ciencias B\'asicas e Ingenier\'ia, Universidad EIA, Envigado, Colombia.\\
\textit{Email address}: \texttt{andres.sarrazola@eia.edu.co}

\medskip
\noindent\textsc{Ra\'ul Vel\'asquez}\\
Instituto de Matem\'aticas, Universidad de Antioquia, Medell\'in, Colombia.\\
\textit{Email address}: \texttt{raul.velasquez@udea.edu.co}

\end{document}